\documentclass[12pt]{amsart}
\usepackage{amssymb, amscd}
\usepackage{color}
\definecolor{darkgreen}{cmyk}{1,0,1,.2}
\definecolor{m}{rgb}{1,0.1,1}
\definecolor{green}{cmyk}{1,0,1,0}

\definecolor{test}{rgb}{1,0,0}
\definecolor{cmyk}{cmyk}{0,1,1,0}

\newtheorem{theorem}{Theorem}[section]
\newtheorem*{thm}{Theorem 1.1}

\newtheorem{proposition}[theorem]{Proposition}
\newtheorem{lemma}[theorem]{Lemma}

\theoremstyle{remark}
\newtheorem{remark}[theorem]{Remark}

\theoremstyle{definition}
\newtheorem{definition}[theorem]{Definition}

\numberwithin{equation}{subsection}

\def\ch{\operatorname{ch}}
\def\Ch{\operatorname{Ch}}

\def\ind{\operatorname{ind}}
\def\End{\operatorname{End}}

\def\Hom{\operatorname{Hom}}

\def\Ker{\operatorname{Ker}}

\def\Tr{\operatorname{Tr}}
\def\STr{\operatorname{Str}}
\def\tr{\operatorname{tr}}

\DeclareMathOperator{\str}{str}
\DeclareMathOperator{\id}{id}

\def\A{\mathbb A}

\def\C{\mathbb C}

\def\R{\mathbb R}

\def\Z{\mathbb Z}

\def\maA{{\mathcal A}}

\def\maD{{\mathcal D}}
\def\maE{{\mathcal E}}
\def\maF{{\mathcal F}}

\def\maL{{\mathcal L}}

\def\maH{{\mathcal H}}

\def\maS{{\mathcal S}}

\def\pa{\partial}

\begin{document}

\title
{Local index theorem for projective families}


\author[M.-T. Benameur]{Moulay-Tahar Benameur}
\address{UMR 7122 du CNRS, Universit\'{e} de Metz, Ile du Saulcy, Metz, France}
\email{benameur@univ-metz.fr}
\author[A. Gorokhovsky ]{Alexander Gorokhovsky}
\address{Department of Mathematics, University of Colorado, Boulder, CO 80305-0395, USA}
\email{Alexander.Gorokhovsky@colorado.edu}

\thanks{The second author was partially supported
by the NSF grant DMS-0900968.
}

\keywords{Key words: Gerbe, Twisted $K$-theory, Family index, Superconnection.}
\subjclass{19L47, 19M05, 19K56.}

\begin{abstract} We give a superconnection proof of the Mathai-Melrose-Singer index theorem for the family of twisted Dirac operators
\cite{MMS1, MMS2}.
\end{abstract}

\maketitle
\tableofcontents
\newpage

\section{Introduction}

Let  $\pi:M\to B$ be a smooth fibration. Given a class $\delta \in H^3(B, \Z)$ Mathai, Melrose
and Singer defined an algebra of twisted by this class  (also called projective) families of pseudodifferential operators.
In order to give this definition in \cite{MMS1} it is assumed that $\delta$ is a torsion class. In \cite{MMS2} the corresponding assumption is that  $\delta = \alpha \cup \beta$, $\alpha \in H^1(B, \Z)$, $\beta \in H^2(B, \Z)$ and $\pi^*\beta =0$.

There is a notion of ellipticity for such a twisted pseudodifferential family and for such an elliptic family
one can then define its index as an element of the $K$-theory of the algebra of smoothing operators.
Mathai, Melrose and Singer than prove in \cite{MMS1, MMS2} an index theorem for such elliptic families
thus giving an extension of the Atiyah-Singer family index theorem \cite{AtiyahSinger4} to the twisted case.
general theory of connections
In this paper we give a superconnection  proof of the cohomological version of this index theorem
for a projective family of Dirac operators. We assume that we are given a gerbe $\maL$ on $B$ with a class
$\delta=\left[\maL\right]\in H^3(B, \Z)$. Our conditions on this class are somewhat weaker than in \cite{MMS1, MMS2};
namely we assume that $\pi^* \delta$ is a torsion element in $H^3(M, \Z)$.
Assume that we are given a $\Z_2$-graded $\maL$ -twisted Clifford module $\maE$ on $M$ as defined in   the Section
\ref{Dirac operators}. We note that such Clifford modules always exist under our assumptions on the class of the gerbe $\maL$. We define the algebra $\Psi_\maL( M|B; \maE)$ of projective families of pseudodifferential operators on $\maE$.
One can then define the index of (the positive part of) the twisted Dirac operator $\ind D^+ \in K_0(\Psi^{-\infty}_\maL( M|B; \maE))$. Here by $\Psi^{-\infty}_\maL( M|B; \maE)$ we denote the algebra of smoothing $\maL$-twisted pseudodifferential operators on $\maE$.

In order to define the Chern character of the index we use cyclic homology and the map constructed in \cite{MathaiStevenson}. This Chern character takes values in the twisted cohomology of $B$ defined as follows.
Let $\varOmega \in \Omega^3(B)$ be a form representing the class $[\maL]$ and let $u$ be a formal variable of degree $-2$.
Twisted cohomology $H^{\bullet}_{\maL}(B)$ is then the homology of the complex $\Omega^*(B)[u]$
with the differential $d_{\varOmega}=ud+u^2\varOmega\wedge \cdot$. Note that the form $\varOmega$ and thus the
complex depend on the choice of connection on $\maL$. Nevertheless the homologies of all the complexes thus obtained
are canonically isomorphic.
Following Mathai and Stevenson \cite{MathaiStevenson}
one  defines a morphism of complexes $\Phi_{\nabla^{\maH}} \colon CC^-_{\bullet}(\Psi^{-\infty}_\maL( M|B; \maE)) \to  \left( \Omega^*(B)[u], d_{\varOmega}\right) $. Here $CC^-_{\bullet}$ denotes the negative cyclic complex.
By composing $\Phi_{\nabla^{\maH}} $  with the Chern character
$\ch \colon K_0(\Psi^{-\infty}_\maL( M|B; \maE)) \to
HC^-_{0}(\Psi^{-\infty}_\maL( M|B; \maE))$ we obtain the class 
$\left[\Phi_{\nabla^{\maH}}(\ch(\ind D^+)) \in H^{\bullet}_{\maL}(B) \right]$. The main result of the paper is the proof of the following theorem expressing the Chern character of the index in terms of characteristic classes:

\begin{theorem}\label{main}\ Let  $D$ be a projective family of Dirac operators on a horizontally $\maL$-twisted Clifford  module $\maE$ on a  fibration  $\pi:M\to B$. Then the following equality holds in $H^{\bullet}_{\maL}(B)$:
$$
\left[\Phi_{\nabla^{\maH}} (\ch (\ind D^+))\right]  = 
\left[u^{-\frac{k}{2}}\int_{M|B}\widehat{A}\left(\frac{u}{2 \pi i}R^{M|B}\right)  \Ch_{\maL}(\maE/\maS)\right],
$$
where $k=\dim M -\dim B$ is the dimension of the fibers.
\end{theorem}
Here, as usual, $\widehat{A}$ is the power series defined by
$\widehat{A}(x)=\left.\det\right.^{1/2}\left(\frac{x/2}{\sinh x/2}\right)$ and $\Ch_{\maL}(\maE/\maS)$ is the 
twisted relative Chern character form of $\maE$, see Proposition \ref{twchern} and equation \eqref{deftwistedchern}.

The characteristic classes appearing in the right hand side are defined in the Section \ref{Dirac operators}.

Our proof uses the Bismut superconnection formalism \cite{Bismut}. We extend the notion of superconnection to the twisted context and show in the Theorem \ref{superconnection index} that the Chern character of superconnection
adapted to $D$ computes the Chern character of $\ind D^+$. Then in the Theorem \ref{local index} using the results of \cite{Bismut} and \cite{BismutFreed} we compute the  limit of the Chern character of the rescaled superconnection obtaining the expression in the right hand side of the index formula. Together these results imply the Theorem \ref{main}.

The paper is organized as follows. In the Section \ref{preliminaries} we give a brief review of cyclic homology,
gerbes and connections on them. We also discuss twisted bundles and their characteristic forms.
In the Section \ref{projective} we give the definitions of the algebra of twisted families of pseudodifferential operators. Finally in the section \ref{Dirac} we define the projective family of Dirac operators on a twisted Clifford module and give the proofs of the main theorems of the paper.

Some work in related direction recently appeared in \cite{CMW, CW, CaW}. Related questions of deformation
theory are considered in \cite{BGNT1,BGNT2}.
 This paper is a byproduct of a joint project with E.~Leichtnam. The authors would like to thank him and
C. Blanchet, M. Karoubi, M.~Lesch, V.~Mathai, R.~Nest and B. Wang for helpful discussions. The authors worked on this paper while visiting
CIRM, Luminy as well as  Hausdorff institute and Max Planck Institute for Mathematics in Bonn. Part of this work was done while the second author was visiting the Laboratoire de Math\'ematiques et Applications of Metz and he is grateful for the hospitality.

\section{Preliminaries}\label{preliminaries}

\subsection{Cyclic homology}

The general reference for this material is the book \cite{Loday}.

Let $A$ be a complex unital algebra. Set $C_k(A)= A\otimes (A/\C1)^{\otimes k}$. Let $u$ be a formal variable of degree $-2$.
The space of negative cyclic chains of degree $ l \in \Z$ is defined by
\[
CC^-_{l}(A) = \left(C_{\bullet}(A)[[u]]\right)_l=\prod \limits_{-2n+k=l,\ n\ge 0} u^nC_k(A).
\]
The boundary is given by $b+uB$ where $b$ and $B$ are the Hochschild and Connes boundaries of the cyclic complex.
The homology of this complex is denoted $HC^-_{\bullet}(A)$.
When the algebra $A$ is $\Z_2$ graded they incorporate the relevant signs.
If $A$ is not necessarily unital denote by $A^+$ its unitalisation and set $CC^-_{l}(A)=CC^-_{l}(A^+)/CC^-_{l}(\C)$.
If $I$ is an ideal in a unital algebra $A$ the relative cyclic complex  is defined by
$CC^-_{\bullet}(A, I) = \Ker \left( CC^-_{\bullet}(A) \to CC^-_{\bullet}(A/I)\right) $. One has a natural morphism of complexes $\iota \colon CC^-_{\bullet}(I) \to CC^-_{\bullet}(A, I)$ induced by the homomorphism $I^+ \to A$.

Recall that for an algebra $A$ we have  Chern character in cyclic homology   $\ch \colon K_0(A)\to HC_{0}^{-}(A)$.
It is  defined by the following formula. Let $P, Q \in M_n(A^+)$ be two idempotents in $n\times n$ matrices of the
algebra $A^+$, representing a class $[P-Q] \in K_0(A)$. Then

\begin{equation}\label{cyclicchern}
{\Ch\left([P-Q]\right) = \tr(P-Q)+ \sum_{n=1}^{\infty} (-u)^n\frac{(2n)!}{n!} \tr\left( \left(P-\frac{1}{2}\right)\otimes P^{\otimes (2n)}-  \left(Q-\frac{1}{2}\right)\otimes Q^{\otimes (2n)} \right)}
\end{equation}
We will use the notation $\Ch\left([P-Q]\right)$ for the cyclic cycle defined above and $\ch\left([P-Q]\right)$ for
its class in cyclic homology $HC_0^-(A)$.

We will also need to use the entire cyclic complex. For our purpose the
algebraic version from  \cite{ConnesBook}, IV.7.$\alpha$ Remark 7 b.
will be sufficient.
First recall that one has the periodic cyclic complex $\left(CC^{per}_{\bullet}(A), b+uB \right) $ where
$CC^{per}_{\bullet}(A) = C_{\bullet}(A)[u^{-1}, u]]$.
Assume we are given a  periodic chain $\alpha = \sum_{k\ge 0} \alpha_k u^k \in CC^{per}_{m}(A) $,  $\alpha_k \in C_{2k+m}(A)$. Then $\alpha$ is called entire if
there exist a finite dimensional subspace $V\subset A$, $1 \in V$ and $C>0$ (depending on $\alpha$) such that $\alpha_k \in V\otimes (V/\C1)^{\otimes k}$ and $\|\alpha_k\| \le  C^k k! $.
Here the norms on $V\otimes (V/\C1)^{\otimes k}$ are induced by an arbitrary
norm on $V$.
We denote the entire cyclic complex of $A$ by $CC^{entire}_{\bullet}(A)$.

Note that the chain $\Ch\left([P-Q]\right)$ defined in \eqref{cyclicchern} is an element in $CC^{entire}_{0}(A)$.

\subsection{Notion of a gerbe}

We give here only the general overview, refering the reader to \cite{Brylinski} and \cite{Hitchin}
for the details. The differential geometry of not necessarily abelian gerbes is described in \cite{BM}.
We will describe the gerbes in terms of their descent data.

Let $M$ be a smooth manifold. Given an open cover $(U_\alpha)_{\alpha\in \Lambda}$ of $M$, we set as usual
$$
U_{\alpha_1\cdots \alpha_k} = \bigcap_{1\leq j \leq k} U_{\alpha_j}.
$$

\begin{definition}
A descent datum for a gerbe $\maL$ on $M$   is  the collection $(U_\alpha, \maL_{\alpha\beta}, \mu_{\alpha\beta\gamma})$ where $(U_\alpha)_{\alpha\in \Lambda}$ is an open cover of $M$, $(\maL_{\alpha\beta} \to U_{\alpha\beta})_{\alpha, \beta\in \Lambda}$ is a collection of line bundles and $\mu_{\alpha\beta\gamma}: \maL_{\alpha\beta} \otimes\maL_{\beta\gamma} \to \maL_{\alpha\gamma}$ are line bundle isomorphisms over each triple intersection $U_{\alpha\beta\gamma}$ such that over each quadruple intersection $U_{\alpha\beta\gamma\delta}$, the following diagram commutes

\[
\begin{CD}
\maL_{\alpha\beta}\otimes \maL_{\beta\gamma}\otimes\maL_{\gamma\delta} 
@>\mu_{\alpha\beta\gamma}\otimes \id>>
\maL_{\alpha\gamma}\otimes\maL_{\gamma\delta}\\
@V \id\otimes\mu_{\beta\gamma\delta}VV  @VV\mu_{\alpha\gamma\delta}V\\
\maL_{\alpha\beta}\otimes \maL_{\beta\delta} @>\ \ \mu_{\alpha\beta\delta}\ \ >>\maL_{\alpha\delta}
\end{CD}
\]
\end{definition}

Notice that we don't assume in this definition  that the open sets $(U_\alpha)_{\alpha\in \Lambda}$
are contractible.

Given two  descent data $ (U_\alpha, \maL_{\alpha\beta}, \mu_{\alpha\beta\gamma})$ and
$ (U_\alpha, \maL_{\alpha\beta}', \mu_{\alpha\beta\gamma}')$ on the same open cover $\{U_{\alpha}\}$
an isomorphism between them is given by line bundles $S_{\alpha}$ on $U_{\alpha}$ and isomorphisms of line
bundles $\lambda_{\alpha \beta} \colon S_{\alpha}^{-1} \otimes \maL_{\alpha \beta} \otimes S_{\beta} \to \maL_{\alpha \beta}'$ over $U_{\alpha \beta}$ so that the diagram

\[
\begin{CD} 
S_{\alpha}^{-1}\otimes \maL_{\alpha \beta} 
\otimes S_{\beta}\otimes S_{\beta}^{-1}\otimes \maL_{\beta \gamma} \otimes S_{\gamma} 
@>\id \otimes \mu_{\alpha \beta \gamma} \otimes \id>>
S_{\alpha}^{-1}\otimes\maL_{\alpha \gamma} \otimes S_{\gamma} \\
@V\lambda_{\alpha \beta}\otimes\lambda_{\beta \gamma}VV @VV\lambda_{\alpha \gamma}V \\
\maL_{\alpha \beta}' \otimes \maL_{\beta \gamma}' @>\ \ \ \mu_{\alpha \beta \gamma}'\ \ \ >>\maL_{\alpha \gamma}'
\end{CD}
\]

commutes.

  Given two
 isomorphisms     $(S_{\alpha}, \lambda_{\alpha \beta})$ and $(S_{\alpha}', \lambda_{\alpha \beta}')$
between  $ (U_\alpha, \maL_{\alpha\beta}, \mu_{\alpha\beta\gamma})$ and
$ (U_\alpha, \maL_{\alpha\beta}',  \mu_{\alpha\beta\gamma}')  $, a two-morphism between them is a collection
of line bundle isomorphisms $\nu_{\alpha} \colon S_{\alpha} \to S_{\alpha}'$ such that
\[
\lambda_{\alpha \beta}' \circ (\nu_{\alpha}^{-1} \otimes \id \otimes \nu_{\beta}) =\lambda_{\alpha \beta}
\]
where we denote by
$\nu_{\alpha}^{-1}$ the isomorphism $S_{\alpha}^{-1} \to (S_{\alpha}')^{-1}$ induced by $\nu_{\alpha}$.

Let $(V_i, \varrho)_{i\in I}$ be a refinement of the open cover  of $(U_\alpha)_{\alpha\in \Lambda}$ of $M$. So $(V_i)_{i\in I}$ is an open cover of $M$ and
$$
\varrho: I \rightarrow \Lambda \quad \text{ such that } \quad V_i \subset U_{\varrho(i)}.
$$
Then restriction to the refinement of $ (U_\alpha, \maL_{\alpha\beta}, \mu_{\alpha\beta\gamma})$ is the descent datum $\maL'=(V_i, \maL'_{ij}, \mu'_{ijk})$ given by:
$$
\maL'_{ij} := \maL_{\varrho(i)\varrho(j)}|_{V_{ij}} \text{ and } \mu'_{ijk} := \mu_{\varrho(i)\varrho(j)\varrho(k)}|_{V_{ijk}}.
$$
Similarly one defines restriction of the isomorphisms and $2$-morphisms to a refinement. We do not
distinguish between a descent datum, isomorphisms of descent data, etc., and  their restrictions to a refinement.
Thus  for instance, the isomorphism between two  descent data $ (U_\alpha, \maL_{\alpha\beta}, \mu_{\alpha\beta\gamma})$ and
$ (U_\alpha', \maL_{\alpha\beta}', \mu_{\alpha\beta\gamma}')$ is an isomorphism between their restrictions to some
common refinement of $\{ U_\alpha\}$ and $\{ U_\alpha'\}$.

An equivalence class of  Dixmier-Douady gerbes  on $M$ is an equivalence class of the descent data on $M$.
More precisely a gerbe is a maximal collection of descent data $\maD_i$, $i \in A$  together with the
isomorphisms $s_{ij}\colon \maD_j \to \maD_i$ for each $i, j \in A$ and $2$-morphisms $\nu_{ijk} \colon
s_{ij} s_{jk} \to s_{ik}$ satisfying the natural associativity condition. We refer the reader
to the book \cite{Brylinski} for the details.

If the cover $U_{\alpha}$ is good \cite{BottTu} all the bundles $\maL_{\alpha \beta}$ are trivializable.
After choice of such trivialization
the collection $(\mu_{\alpha\beta\gamma})$ can be viewed  as a \v{C}ech $2$-cochain with coefficents in the sheaf $\underline{\C^*}$ of smooth functions with values in the nonzero complex numbers $\C^*$.
The compatibility  condition over $U_{\alpha\beta\gamma\delta}$ tells us that $\mu$ is a $2$-cocycle and hence defines a cohomology class  $[\mu]\in H^2(M; \underline{\C^*}) \cong H^3(M, \Z)$. This class is a well defined invariant of the gerbe called the Dixmier-Douady class. We denote this class by $[\maL]$.  Every class in $H^3(M, \Z)$ is a class of a gerbe defined by this class uniquely up to an isomorphism  (see \cite{Brylinski}).

Given a smooth map $f: M'\to M$ between smooth manifolds $M'$ and $M$, we can pull-back any decent datum for a gerbe on $M$ to a descent datum on $M'$. The pull backs of isomorphic descent data are isomorphic and thus we obtain a well-defined pull-back of a gerbe. Clearly the Dixmier-Douady class of the pull-back is the pull-back of the Dixmier-Douady class.

An unitary descent datum is  $(U_\alpha, \maL_{\alpha\beta},  \mu_{\alpha\beta\gamma})$ together with a choice of metric on each $\maL_{\alpha\beta}$  such that each
$\mu_{\alpha\beta\gamma}$ is an isometry.
A notion of unitary equivalence of two unitary descent data on the same
open cover $U_{\alpha}$ is obtained from the notion of equivalence above by requiring that each
line bundle $S_{\alpha}$ is Hermitian and each $\lambda_{\alpha \beta}$ is an isometry.  The definition
of $2$-morphisms is modified by requiring each $\nu_{\alpha}$ to be an isometry.
It is clear that the restriction of a unitary descent datum to a refinement is again unitary.
Then a unitary gerbe is an equivalence class of unitary descent data in the sense described above.

\subsection{Connections on gerbes}

\begin{lemma}\label{descconnection}
Let  $(U_\alpha, \maL_{\alpha\beta}, \mu_{\alpha\beta\gamma})$  be a descent datum on $M$.  There exists  a collection $(\nabla_{\alpha\beta})$ of connections on $(\maL_{\alpha\beta})$ such that for any $(\alpha, \beta, \gamma)\in \Lambda^3$ with $U_{\alpha \beta \gamma} \ne \emptyset$:
$$
\mu_{\alpha\beta\gamma}^* \nabla_{\alpha\gamma} = \nabla_{\alpha\beta}\otimes \id + \id \otimes \nabla_{\beta\gamma}.
$$
If the descent datum is unitary  each $\nabla_{\alpha\beta}$ can be chosen Hermitian.
\end{lemma}

\begin{proof}
Fix for any $\alpha, \beta$ with $U_{\alpha\beta}\not = \emptyset$ a connection $\nabla'_{\alpha\beta}$ on $\maL_{\alpha\beta}$. We set for $U_{\alpha\beta\gamma}\not =\emptyset$,
$$
A_{\alpha\beta\gamma} := \mu_{\alpha\beta\gamma}^* \nabla_{\alpha\gamma} - \nabla_{\alpha\beta}\otimes \id - \id \otimes \nabla_{\beta\gamma}.
$$
Then using the identification
$$
\End(\maL_{\alpha\beta}\otimes \maL_{\beta\gamma}) \simeq U_{\alpha\beta\gamma}\times \C;
$$
we see that $A_{\alpha\beta\gamma}$ is identified with a differential $1$-form on the  open set $U_{\alpha\beta\gamma}$. We also have for $U_{\alpha\beta\gamma\delta} \not = \emptyset$
$$
A_{\beta\gamma\delta} + A_{\alpha\beta\delta} = A_{\alpha\gamma\delta} + A_{\alpha\beta\gamma}.
$$
Therefore, there exists $A'=(A'_{\alpha\beta})$ such that
$$
A_{\alpha\beta\gamma} = A'_{\beta\gamma} - A'_{\alpha\gamma} + A'_{\alpha\beta}.
$$
The collection of connections $(\nabla_{\alpha\beta}=\nabla'_{\alpha\beta} + A'_{\alpha\beta})$ is then a connection on the gerbe $\maL$.
\end{proof}

\begin{lemma}
Let $(\nabla_{\alpha\beta})$ be as above, and denote by $\omega_{\alpha\beta}=\nabla^2_{\alpha\beta}$ the curvatures
of the connections $\nabla_{\alpha \beta}$. Then
there exists a collection   of differential $2$-forms $\omega_\alpha\in \Omega^2(U_\alpha)$ such that
$$
\omega_{\alpha\beta}=\omega_{\alpha} - \omega_{\beta}, \quad \text{ for } U_{\alpha\beta} \not = \emptyset.
$$
\end{lemma}

\begin{proof}
 The collection $\omega_{\alpha\beta}$  satisfies
$$
\omega_{\alpha\gamma} = \omega_{\alpha\beta} + \omega_{\beta\gamma}.
$$

This shows that $(\omega_{\alpha\beta})$ is a \v{C}ech cocycle and hence by the  acyclicity of the \v{C}ech complex of forms, there exist $(\omega_\alpha)$ such that
$$
\omega_{\alpha\beta} = \omega_\alpha - \omega_\beta.
$$
\end{proof}

We will say that  the collection $(\nabla_{\alpha \beta}, \omega_{\alpha})$ is a connection on the descent
datum $(U_\alpha, \maL_{\alpha\beta}, \mu_{\alpha\beta})$. 

 Connection on descent datum yields a connection on its restriction to a refinement
in an obvious manner; we will identify connection with its restriction. 
An isomorphism between  descent datum $(U_\alpha, \maL_{\alpha\beta}, \mu_{\alpha\beta})$  with connection
$(\nabla_{\alpha \beta}, \omega_{\alpha})$ and descent datum  $(U_\alpha, \maL_{\alpha\beta}', \mu_{\alpha\beta}')$
with connection $(\nabla_{\alpha \beta}', \omega_{\alpha}')$   is given by
$(S_{\alpha}, \lambda_{\alpha \beta}, \nabla_{\alpha} )$ where $(S_{\alpha}, \lambda_{\alpha \beta} ) $ is an isomorphsm
betweeween the descent data (without connections) and each
$\nabla_{\alpha}$ is a connection on
$S_{\alpha}$ satsfying the following conditions.  
Let $\pi_\alpha =\nabla_{\alpha}^2$ be the curvatures of these connections and 
let $\nabla_{\alpha}^{-1}$ be the dual connection on $S_{\alpha}^{-1}$.  
Then we require 
the following dentities (using the isomorphism 
$\maL_{\alpha \beta} = S_{\alpha}\otimes ( S_{\alpha}^{-1}\otimes \maL_{\alpha \beta} \otimes S_{\beta})\otimes S_{\beta}^{-1}$):
\begin{equation}\label{nabla'}
\nabla_{\alpha \beta} = (\lambda_{\alpha \beta})^* (\nabla_{\alpha} \otimes \id \otimes \id +\id \otimes \nabla_{\alpha \beta}' \otimes \id +\id \otimes \id \otimes \nabla_{\beta}^{-1})
\end{equation}
and
\begin{equation}\label{omega'}
\omega_{\alpha} =\omega_{\alpha}'+\pi_{\alpha}.
\end{equation}

If $s =(S_{\alpha}, \lambda_{\alpha \beta}, \nabla_{\alpha} )$ and  $s' =
(S_{\alpha}', \lambda_{\alpha \beta}', \nabla_{\alpha}' )$ 
are  two isomorphisms between the descent data with connections 
as above, the $2$-morphisms between $s$ and $s'$ are the same as the $2$ morphisms
between $(S_{\alpha}, \lambda_{\alpha \beta})$ and  $
(S_{\alpha}', \lambda_{\alpha \beta}')$. 
 
Assume $\maL$ is a gerbe given by  a  collection of descent data $\maD_i$, $i \in A$,
isomorphisms $s_{ij}\colon \maD_j \to \maD_i$ for each $i, j \in A$ and $2$-morphisms $\nu_{ijk} \colon
s_{ij} s_{jk} \to s_{ik}$.
A connection on a gerbe $\maL$ is a lift of each $\maD_i$ to a descent datum with connection $\widetilde{\maD_i}$
and each $s_{ij}$ to an isomorphsm $\widetilde{s_{ij}} \colon \widetilde{\maD_j} \to \widetilde{\maD_i}$.
 An easy argument using the Lemma \ref{descconnection} shows that on a given gerbe always exists a connection.

\begin{lemma}\label{defomega} Let $\maL$ be a gerbe represented by a descent
datum $(U_\alpha, \maL_{\alpha\beta}, \mu_{\alpha\beta\gamma})$.
Choose a connection on $\maL$ represented by a connection $(\nabla_{\alpha \beta}, \omega_{\alpha})$
on the descent datum.
\begin{enumerate}
\item There exists a well-defined closed form $\varOmega\in \Omega^3(M)$ --
curvature $3$-form of the connection --
such that $\varOmega|_{U_{\alpha}}=\frac{d\omega_{\alpha}}{2\pi i}$.

\item  Let $(\nabla_{\alpha \beta}', \omega_{\alpha}')$ be  another
connection on $\maL$ and let $\varOmega'$ be the corresponding curvature $3$-form.
Then there exists a canonical  $\eta \in \Omega^2(M)/d\Omega^1(M)$ such that
$\varOmega'= \varOmega+ d\eta$.

\item Let $(\nabla_{\alpha \beta}, \omega_{\alpha})$ $(\nabla_{\alpha \beta}', \omega_{\alpha}')$,
$(\nabla_{\alpha \beta}'', \omega_{\alpha}'')$ be $3$  connections on $\maL$ with the corresponding
curvature $3$-forms $\varOmega$, $\varOmega'$, $\varOmega''$. Let $\eta, \eta', \eta'' \in
\Omega^2(M)/d\Omega^1(M)$ be the canonical elements constructed above such that $
\varOmega'-\varOmega=d \eta$, $\varOmega''-\varOmega'=d\eta'$, $\varOmega''-\varOmega=d\eta''$.
Then $\eta''=\eta+\eta'$.
\end{enumerate}
\end{lemma}
The $3$-form $\varOmega$ is a  de Rham representative of the Dixmier-Douady class of the gerbe.
\begin{proof}
We use notations of the previous Lemma.
For the first part notice that since $d\omega_{\alpha\beta}=0$ for any $\alpha, \beta$ we see that
$$
d\omega_\alpha|_{U_{\alpha\beta}} = d\omega_\beta|_{U_{\alpha\beta}}
$$
which shows the existence of $\varOmega$; it is clearly closed.
$\varOmega$ does not change if one restricts the descent datum to a refinement.
The equation \eqref{omega'} implies that if one uses  a different descent datum  
for the same gerbe one obtains the same $3$-curvature form.

We proceed to the proof of  the second part. Let $(\nabla_{\alpha \beta}, \omega_{\alpha})$ and
$(\nabla_{\alpha \beta}', \omega_{\alpha}')$ be two connections on the descent datum
 $(U_\alpha, \maL_{\alpha\beta}, \mu_{\alpha\beta\gamma})$.
Set $\delta_{\alpha \beta}=\nabla_{\alpha\beta}'-\nabla_{\alpha\beta}\in \Omega^1(U_{\alpha \beta})$. Choose $\delta_{\alpha} \in \Omega^1(U_{\alpha})$ such that $\delta_{\alpha \beta}=\delta_{\alpha}-\delta_{\beta}$ on $U_{\alpha\beta}$.
Then $\omega_{\alpha \beta}'-\omega_{\alpha \beta} = d\delta_{\alpha\beta} = d\delta_{\alpha}-d \delta_{\beta}$. Hence $\omega_{\alpha}'-\omega_{\alpha} -d \delta_{\alpha}= \omega_{\beta}'-\omega_{\beta} -d \delta_{\beta}$ on $U_{\alpha\beta}$.
Therefore there exists $\eta \in \Omega^2(M)$
such that $\eta|_{U_{\alpha}}=\frac{1}{2\pi i}(\omega_{\alpha}'-\omega_{\alpha} -d \delta_{\alpha})$.
Then $\varOmega'-\varOmega=d\eta$.
The formula for $\eta$ given above depends on the choice of $\delta_{\alpha}$.
 If $\bar{\delta}_{\alpha}$ is a different such choice, then $\bar{\delta}_{\alpha} - \delta_{\alpha}=\epsilon|_{U_{\alpha}}$ for some $\epsilon \in \Omega^1(M)$.
 Then  $\bar{\eta}=
\eta - d\epsilon$.
 Hence the class of $\eta \in \Omega^2(M)/d\Omega^1(M)$ does not depend on the choice made.
It is easy to see that $\eta$ does not depend on the choice of the particular descent datum used.

The verification of the third statement is straightforward and is left to the reader.

\end{proof}

\subsection{Twisted cohomology}
For a smooth manifold $M$ let $\varOmega \in \Omega^3(M)$ be a closed $3$-form. Denote
by $u$ a formal variable of degree $-2$. The twisted de Rham complex is defined as the
complex $\Omega^*(M)[u]$ with the differential $d_{\varOmega}=ud +u^2\varOmega\wedge \cdot$.
Note that if $\varOmega'=\varOmega+d\eta $ is cohomologous to $\varOmega$ then the complexes $\left(\Omega^*(M)[u], d_{\varOmega} \right)
$  and $\left(\Omega^*(M)[u], d_{\varOmega'} \right)
$ are isomorphic via the isomorphism
\begin{equation}\label{defI}
I_{\eta} \colon \xi \mapsto e^{-u\eta}\wedge \xi
\end{equation}

\begin{lemma}\label{Imap}
The map induced by $I_{\eta}$ on cohomology depends only on the class of $\eta$ in
$\Omega^2(M)/\left(d\Omega^1(M)\right)$.
\end{lemma}
\begin{proof}
 Indeed, for $\epsilon \in \Omega^1(M)$ define
$h_{\epsilon} \colon \left( \Omega^*(M)[u]\right)_{\bullet} \to \left( \Omega^*(M)[u]\right)_{\bullet-1}$
by
\[h_{\epsilon}= u\epsilon \left( \sum_k \frac{(u\eta)^k}{(k+1)!}\right)~\wedge~\cdot.
\]
Then
\[I_{\eta+d\epsilon}-I_{\eta}=h_{\epsilon}\circ d_{\varOmega}+d_{\varOmega'}\circ h_{\epsilon}
\]
and therefore the maps $I_{\eta}$ and $I_{\eta+d\epsilon}$ are chain homotopic.
\end{proof}

The identity $d_{\varOmega+\varOmega'} (\xi \wedge \eta) =  d_{\varOmega}\xi \wedge \eta + (-1)^{|\xi|} \xi \wedge d_{\varOmega'}  \eta$ implies that the product of forms induces the product $H^{\bullet}_{\varOmega}(M) \otimes H^{\bullet}_{\varOmega'}(M) \to H^{\bullet}_{\varOmega+\varOmega'}(M)$ and in particular endows
 $H^{\bullet}_{\varOmega}(M)$ with the structure of   $H^{\bullet}(M)[u]$-module.

We will also need to consider the following situation. Let $\pi \colon M \to B$ be an oriented fibration with compact fibers.
Then we have an integration along the fibers map $ \int_{M|B}\colon \Omega^*(M) \to \Omega^{*-k}(B)$, $k=\dim M - \dim B$. Let $\varOmega \in \Omega^3(B)$ be a closed form. Then $ \int_{M|B} d \eta = d \int_{M|B} \eta$, $\int_{M|B} \pi^*\varOmega \wedge \eta = \varOmega\wedge \int_{M|B} \eta$ and therefore
\[
\int_{M|B} d_{\pi^*\varOmega} \eta = d_{\varOmega} \int_{M|B} \eta.
\]
Hence we obtain a chain map $\int_{M|B} \colon \left(\Omega^*(M), d_{\pi^* \varOmega} \right)_{\bullet} \to
\left(\Omega^*(B), d_{\varOmega} \right)_{\bullet-k}$.

Let $\maL$ be a gerbe. The choice of connection defines a closed  $3$-form $\varOmega$.
We can therefore consider the complex $\left(\Omega^*(M)[u], d_{\varOmega} \right)$.
Different choice of connection leads to a different complex $\left(\Omega^*(M)[u], d_{\varOmega'} \right)$.
Lemma \ref{defomega} however implies that there exists however a canonical isomorphism
$I$ of the homologies of these complexes. We denote this homology $H^{\bullet}_{\maL}(B)$


\subsection{Twisted bundles }

\begin{definition}
A descent datum for a twisted vector bundle $\maE$ consists of a descent datum
$(U_{\alpha}, \maL_{\alpha\beta}, \mu_{\alpha\beta\gamma})$ for a gerbe $\maL$ together with a collection
$(\maE_\alpha\to U_\alpha)_{\alpha\in \Lambda}$ of vector bundles  and  a collection of vector bundle isomorphisms $\varphi_{\alpha\beta}: \maE_\alpha\otimes \maL_{\alpha\beta} \cong \maE_\beta$ such that for every $\alpha$, $\beta$, $\gamma$ the following diagram commutes

\[
\begin{CD}  
\maE_\alpha\otimes \maL_{\alpha\beta}\otimes \maL_{\beta\gamma} @>\id \otimes \mu_{\alpha\beta\gamma}>>
\maE_\alpha \otimes \maL_{\alpha\gamma}\\
@V\varphi_{\alpha\beta}\otimes \id VV    @VV\varphi_{\alpha\gamma}V \\
\maE_\beta\otimes \maL_{\beta\gamma} @>\ \ \varphi_{\beta\gamma}\ \ >> \maE_{\gamma}
\end{CD}
\]

\end{definition}

Restriction of the descent datum for $\maE$ to a refinement is given by the restriction of the descent datum
for $\maL$ together with restriction of the vector bundles $\maE_{\alpha}$.

\begin{definition}
An (iso)morphism between two descent data $( U_{\alpha}, \maL_{\alpha\beta}, \mu_{\alpha\beta\gamma}, \maE_{\alpha}, \varphi_{\alpha \beta})$ and
$(  U_{\alpha}', \maL_{\alpha\beta}', \mu_{\alpha\beta\gamma}', \maE_{\alpha}', \varphi_{\alpha \beta}')$ is given by the collection  $(\rho_{\alpha}, S_{\alpha}, \lambda_{\alpha \beta})$ where $( S_{\alpha}, \lambda_{\alpha \beta})$ is an isomorphism between
$(  U_{\alpha}, \maL_{\alpha\beta}, \mu_{\alpha\beta\gamma})$ and
$( U_{\alpha}', \maL_{\alpha\beta}', \mu_{\alpha\beta\gamma}')$ and $\rho_{\alpha} \colon \maE_{\alpha} \otimes S_{\alpha} \to \maE_{\alpha}'$ is a collection of (iso)morphisms such that the diagram
\[
\begin{CD}
\maE_{\alpha} \otimes S_{\alpha}\otimes S_{\alpha}^{-1}\otimes \maL_{\alpha \beta} \otimes S_{\beta}
@>\varphi_{\alpha \beta} \otimes \id>>
\maE_{\beta} \otimes S_{\beta}\\
@V\rho_{\alpha}\otimes\lambda_{\alpha \beta }VV @VV\rho_{\beta}V\\
\maE_{\alpha }' \otimes \maL_{\alpha \beta }' @>\ \varphi_{\alpha \beta \gamma}'\ >> \maE_{\beta}'
\end{CD}
\]

commutes.
\end{definition}
An isomorphism  between two descent data on two different covers is defined as an isomorphism
between their restriction on a common refinement. A $2$-morphism between two isomorphisms is
the $2$-morphisms between the corresponding isomorphisms of the gerbe descent data.

A twisted bundle is then defined as an equivalence  class
of descent data of twisted vector bundles. ``Forgetting'' the bundle data we obtain from the descent datum for a twisted vector
bundle a descent datum for a gerbe, and the same applies to morphisms and $2$-morphisms.
We say that
$\maE$ is an $\maL$-twisted vector bundle if ``forgetting'' the bundle data one obtains the equivalence class
of the gerbe descent data defining $\maL$.


Assume now that the gerbe $\maL$ is unitary. An Hermitian descent datum for $\maE$  consists of a unitary descent datum $(U_{\alpha}, \maL_{\alpha \beta}, \mu_{\alpha \beta \gamma})$ for $\maL$   and a collection $h_{\alpha}$ of metrics on $\maE_{\alpha}$ such that the maps $\varphi_{\alpha \beta}$ are isometries. One obtains a notion of isomorphism of Hermitian descent data by requiring $\rho_{\alpha}$ to be isometries. An Hermitian twisted bundle is then an equivalence class of Hermitian descent data.

Given a gerbe $(U_{\alpha}, \maL_{\alpha\beta}, \mu_{\alpha\beta\gamma})$ on $M$, it is well known that 
 (finite dimensional) twisted vector bundles exist if and only if the gerbe is torsion 
(see e.g. \cite{DK, Karoubi} and references therein).


\begin{lemma} \label{extend} Let $\maL$ be a gerbe on $M$ and $\maE$ an $\maL$-twisted bundle. Let $(U_{\alpha}, \maL_{\alpha \beta},
\mu_{\alpha \beta \gamma})$ be a descent datum for $\maL$. Then there exists a descent datum for $\maE$ isomorphic
to the one of the form
$(U_{\alpha}, \maL_{\alpha \beta}, \mu_{\alpha \beta \gamma}, \maE_{\alpha}, \varphi_{\alpha \beta})$.
\end{lemma}
\begin{proof}
Let $(V_{i}, L_{i j}, m_{ijk}, E_{i}, r_{ij}), i,j,k \in I$ be a descent datum for $\maE$. Without a
loss of generality we may assume that $\{V_i\}$ is refinement of $\{U_{\alpha}\}$ given by the map  $\varrho: I\to \Lambda$ and that there exists an isomorphism $(S_i, \lambda_{i j}) $ between $(V_{i}, L_{i j}, m_{ijk} )$
and restriction of $(U_{\alpha}, \maL_{\alpha\beta}, \mu_{\alpha\beta\gamma})$ to $\{V_i\}$.

We define the vector bundle $\maE_\alpha$ as follows. On every non empty open set $U_\alpha\cap V_i$, we set
$$
\maE_\alpha^{(i)} := E_i \otimes   S_i \otimes \maL_{\varrho(i)\alpha}.
$$
Notice that $V_i\cap U_\alpha\subset U_{\varrho(i)\alpha}$ so that this definition makes sense as a vector bundle over $V_i\cap U_\alpha$. Moreover, if $(i,j)\in I^2$ is such that $V_{ij}\cap U_\alpha \not = \emptyset$, then we have a  bundle isomorphism
$$
\psi_\alpha^{(ij)}\colon\maE_\alpha^{(j)}|_{V_{ij}\cap U_\alpha} \longrightarrow \maE_\alpha^{(i)}|_{V_{ij}\cap U_\alpha}.
$$
defined by the composition
\begin{multline*}
E_j \otimes   S_j \otimes \maL_{\varrho(j)\alpha} \overset{\varphi_{ij}^{-1}\otimes \id}\longrightarrow E_i \otimes L_{ij} \otimes   S_j \otimes \maL_{\varrho(j)\alpha} \longrightarrow E_i \otimes   S_i \otimes   S_i^{-1} \otimes L_{ij} \otimes   S_j \otimes \maL_{\varrho(j)\alpha}\\ \overset{\id \otimes \lambda_{ij} \otimes \id} \longrightarrow
E_i \otimes   S_i \otimes \maL_{\varrho(i)\varrho(j)} \otimes \maL_{\varrho(j)\alpha} \overset{\id \otimes \mu_{\varrho(i) \varrho(j) \alpha}}\longrightarrow
E_i \otimes   S_i  \otimes \maL_{\varrho(i)\alpha}
\end{multline*}

It is easy to see  that
$$
\psi_\alpha^{(ik)} = \psi_\alpha^{(ij)} \circ \psi_\alpha^{(jk)}.
$$
So, we can glue the bundles $(\maE_\alpha^{(i)})_i$ for $V_i\cap U_\alpha\not=\emptyset$ together and form a vector bundle $\maE_\alpha$ over $U_\alpha$. More precisely
$$
\maE_\alpha := \left(\amalg_{V_i\cap U_\alpha\not=\emptyset} \maE_\alpha^{(i)}\right) / \{\psi_{\alpha}^{(ij)}\}.
$$
Next we introduce for $i\in I$ with $V_i\cap U_{\alpha\beta}\not=\emptyset$,
$$
\varphi_{\alpha\beta}^{(i)} :=   \id \otimes \mu_{\varrho(i)\alpha\beta} \colon \maE_\alpha^{(i)}\otimes \maL_{\alpha\beta} \rightarrow \maE_\beta^{(i)}.
$$
When $V_{ij}\cap U_{\alpha\beta} \not=\emptyset$, a straightforward inspection shows again that the following diagram commutes

\[
\begin{CD} 
\maE_\alpha^{(i)}\otimes \maL_{\alpha\beta} 
@>\ \ \varphi^{(i)}_{\alpha \beta}\ \ >>
\maE_\beta^{(i)}\\
@V\psi_{\alpha}^{(ji)}\otimes \id VV @VV\psi_{\beta}^{(ji)} V \\
\maE_\alpha^{(j)}\otimes \maL_{\alpha\beta}
@>\ \ \varphi^{(j)}_{\alpha \beta}\ \ >>
\maE_\beta^{(j)}
\end{CD}
\]

Therefore, the isomorphisms $\varphi_{\alpha\beta}^{(i)}$ induce an isomorphism
$$
\varphi_{\alpha\beta}: \maE_\alpha\otimes \maL_{\alpha\beta} \rightarrow \maE_\beta.
$$
We leave it to the reader to show that   $(U_{\alpha}, \maL_{\alpha\beta}, \mu_{\alpha\beta\gamma}, \maE_{\alpha}, \varphi_{\alpha \beta})$ is a descent datum for a twisted  vector bundle isomorphic to $(V_{i}, L_{i j}, m_{ijk}, E_{i}, r_{ij})$.


\end{proof}

We now fix a gerbe $\maL$ with a descent datum $(U_\alpha, \maL_{\alpha\beta}, \mu_{\alpha\beta\gamma})$ on the smooth manifold $M$ together with a twisted vector bundle $\maE$ represented by $(\maE_\alpha, \varphi_{\alpha\beta})$. We denote by $\maA_\alpha$ the collection of bundles of algebras
$$
\maA_\alpha := \End( \maE_\alpha), \quad \alpha\in \Lambda.
$$
For any $U_{\alpha\beta}\not = \emptyset$, we have a canonical vector bundle isomorphism over $U_{\alpha\beta}$
$$
\rho_{\alpha\beta} : \End( \maE_\alpha \otimes \maL_{\alpha\beta}) \longrightarrow \maA_\alpha,
$$
extending the canonical isomorphism $\End(\maL_{\alpha\beta}) \simeq U_{\alpha\beta}\times \C$. Therefore, the bundle isomorphism $\varphi_{\alpha\beta}$ together with the identification $\rho_{\alpha\beta}$ induce the isomorphism of algebra bundles over $U_{\alpha\beta}$ given by
$$
\maA_\beta \stackrel{\varphi_{\alpha\beta}^*}{\longrightarrow} \End( \maE_\alpha \otimes \maL_{\alpha\beta})\stackrel{\rho_{\alpha\beta}^*}{\longrightarrow} \maA_\alpha.
$$
We denote  this isomorphism by $\varphi_{\alpha\beta}^*$. It is then easy to check that
$$
\varphi_{\alpha\beta}^* \circ \varphi_{\beta\gamma}^* = \varphi_{\alpha\gamma}^*, \quad \text{ over } U_{\alpha\beta\gamma}.
$$
Therefore, the collection $(\maA_\alpha)$   defines a bundle $\maA $ of algebras over $M$, which we denote $\End(\maE)$. We leave to the reader an easy check that the isomorphism class of $\End(\maE)$ depends only on the isomorphism class of $\maE$.

Note that for every $\alpha$ we have the trace $ \tr_{\alpha} \colon \End(\maE_{\alpha}) \to C^{\infty}(U_{\alpha})$.
For $a \in \Gamma(U_{\alpha \beta}; \maA_{\beta})$ $\tr_{\alpha}(\varphi_{\alpha \beta}^*(a))= \tr_{\beta}(a)$.
Therefore we obtain the trace $\tr \colon \End(\maE) \to C^{\infty}(M)$ defined by
\[ \tr(a)|_{U_{\alpha}}= \tr_{\alpha}(a|_{U_{\alpha}}) \text{ for } a \in \End(\maE)
\]
If the bundle $\maE$ is $\Z_2$-graded then $\End(\maE)$ is a bundle
of $\Z_2$-graded algebras with a supertrace  $ \str \colon \End(\maE) \to C^{\infty}(M)$.

\begin{definition}
The bundle $\maA = \End(\maE) $ is called the Azumaya bundle associated with the $\maL$-twisted bundle $\maE$.
\end{definition}

If  $\maE$ is an Hermitian twisted bundle, each of the bundles $\End(\maE_{\alpha})$ is a bundle of *-algebras,
with the *-operation given by taking the adjoint endomorphism. This induces a structure of a bundle of *-algebras on $\End(\maE)$.

Notice that, more generally, if $\maE$ and $\maE'$ are $\maL$-twisted bundles, we have a well defined
bundle $\Hom (\maE, \maE')$, defined similarly by $\left.\Hom (\maE, \maE') \right|_{U_{\alpha}}=\Hom (\maE|_{U_{\alpha}}, \maE'|_{U_{\alpha}})$.

Let $\maL$ be a gerbe  on $M$ and $\maE$ an $\maL$-twisted bundle on $M$. Let $(U_\alpha, \maL_{\alpha\beta}, \mu_{\alpha\beta\gamma}, \maE_\alpha, \varphi_{\alpha\beta}) $ be a descent datum for $\maE$.

\begin{definition}
A connection on $(U_\alpha, \maL_{\alpha\beta}, \mu_{\alpha\beta\gamma}, \maE_\alpha, \varphi_{\alpha\beta}) $  is a collection $(\nabla_\alpha, \nabla_{\alpha \beta}, \omega_{\alpha})$ where $(\nabla_{\alpha \beta}, \omega_{\alpha})$ is a connection on $(U_\alpha, \maL_{\alpha\beta}, \mu_{\alpha\beta\gamma})$ and each $\nabla_{\alpha}$ is a connection on $\maE_{\alpha}$ such that the identities

\begin{equation}\label{connectionon}
\varphi_{\alpha\beta}^* \nabla_\beta = \nabla_\alpha\otimes \id + \id \otimes \nabla_{\alpha\beta}
\end{equation}
$\text{hold for } U_{\alpha\beta}\not = \emptyset.$
\end{definition}

\begin{lemma}\label{connection existence}
Let $(U_\alpha, \maL_{\alpha\beta}, \mu_{\alpha\beta\gamma}, \maE_\alpha, \varphi_{\alpha\beta}) $ be a descent datum for an $\maL$-twisted bundle $\maE$.
Then every connection $(\nabla_{\alpha\beta}, \omega_{\alpha})$ on the descent datum  $(U_\alpha, \maL_{\alpha\beta}, \mu_{\alpha\beta\gamma})$ for $\maL$ can be extended to a connection for the descent datum for $\maE$.
\end{lemma}

\begin{proof}
We start with a collection $(\nabla'_\alpha)$ of connections on $(\maE_\alpha)$. Then we set for $U_{\alpha\beta}\not = \emptyset$:
$$
A_{\alpha\beta} = \varphi_{\alpha\beta}^* \nabla'_\beta - \nabla'_\alpha\otimes \id - \id \otimes \nabla_{\alpha\beta}
$$
So, $A_{\alpha\beta}$ is differential $1$-form on $U_{\alpha\beta}$ with coefficients in the bundle $\End(\maE_\alpha\otimes \maL_{\alpha\beta})$. This latter being canonically isomorphic to $\End(\maE_\alpha)$, we see that
$$
A_{\alpha\beta} \in \Omega^1 (U_{\alpha\beta}, \End( \maE_\alpha)).
$$
Using the functoriality conditions on the isomorphisms $(\varphi_{\alpha\beta})$, it is then easy to check that
$$
A_{\alpha\gamma} = A_{\alpha\beta} + \varphi_{\alpha\beta}^* A_{\beta\gamma},
$$
where $\varphi_{\alpha\beta}$ is viewed here as the isomorphism over $U_{\alpha\beta}$ between $\End(\maE_\beta)$ and $\End(\maE_\alpha)$. Since  the sheaf of sections of the bundle $\End(\maE)$ is soft, there exists a collection $ A_\alpha\in \Omega^1(U_\alpha; \End(\maE_\alpha)) $ such that
$$
 A_{\alpha\beta} = A_\alpha - \varphi_{\alpha\beta}^*A_\beta.
$$
The collection $\nabla_\alpha= \nabla'_\alpha + A_\alpha$ is then a connection on the $\maL$-twisted vector bundle $\maE$ compatible with the curving $\nabla_{\alpha \beta}$.
\end{proof}


An  isomorphism  between two descent 
data   $( U_{\alpha}, \maL_{\alpha\beta}, \mu_{\alpha\beta\gamma}, \maE_{\alpha}, \varphi_{\alpha \beta})$ and
$(  U_{\alpha}', \maL_{\alpha\beta}', \mu_{\alpha\beta\gamma}', \maE_{\alpha}', \varphi_{\alpha \beta}')$ 
for $\maE$ with connections $(\nabla_\alpha, \nabla_{\alpha \beta}, \omega_{\alpha}) $ and
$(\nabla_\alpha', \nabla_{\alpha \beta}', \omega_{\alpha}') $ respectively is given by the collections
 $s=(\rho_{\alpha}, S_{\alpha}, \lambda_{\alpha \beta}, )$ where 
$(\rho_{\alpha}, S_{\alpha}, \lambda_{\alpha \beta})$ is a morphism between the descent data 
without connections,  $\nabla_\alpha^S$ are connections on $S_\alpha$ such that 
$\left(S_{\alpha}, \lambda_{\alpha \beta}, \nabla_\alpha^S\right)$ is a morphism of the corresponding 
gerbe descent data with connectictions (i.e. the equatons 
\eqref{nabla'}, \eqref{omega'} are satisfied) and the equality 
\[
\rho_{\alpha}^* \nabla_{\alpha}' =\nabla_{\alpha} \otimes \id + \id \otimes \nabla_{\alpha}^S
 \]
holds (in the notations of \eqref{nabla'} and \eqref{omega'}). 

A connection on a twisted bundle is then a choice  of connections on each
descent datum of this twisted bundle and lifting of  isomorphisms of descent data to
isomorphism of descent data with connections. Every connection
on a gerbe $\maL$ can be extended to a connection on any $\maL$-twisted bundle.

\begin{proposition}\label{global curvature}
Let $\maE$ be an $\maL$-twisted bundle with connection. Choose a descent datum $(U_\alpha, \maL_{\alpha\beta}, \mu_{\alpha\beta\gamma}, \maE_\alpha, \varphi_{\alpha\beta}) $   with a connection $(\nabla_\alpha, \nabla_{\alpha \beta}, \omega_{\alpha})$ representing $\maE$. Then the collection $(\theta_\alpha + \omega_\alpha)$, where $\theta_\alpha=\nabla_\alpha^2$ is the curvature of $\nabla_{\alpha}$, defines a global differential $2$-form $\theta$ on $M$ with coefficients in the Azumaya bundle $\maA=\End(\maE)$. This form is independent of the choice of the representing descent datum.
\end{proposition}

\begin{proof}
We have for any $\alpha\in \Lambda$, $\theta_\alpha\in \Omega^2 (U_\alpha, \maA_\alpha)$ where $\maA_\alpha=\End(\maE_\alpha))$ is the Azumaya bundle associated with the twisted bundle $\maE$.  The equation \eqref{connectionon} implies that
$$
\varphi_{\alpha\beta}^* \theta_\beta = \theta_\alpha +  \omega_{\alpha\beta}\quad \in \quad \Omega^2(U_{\alpha\beta}, \End(\maE_\alpha)).
$$%
Therefore, the collection $(\theta_\alpha + \omega_\alpha)$ of elements of $\Omega^2(U_\alpha, \maA_\alpha)$ satisfies the relations
$$
\varphi_{\alpha\beta}^*(\theta_\beta + \omega_\beta) = \theta_\alpha + \omega_\alpha.
$$
It is easy to see that the form $\theta$ is independent of the equivalence class of the connection
and is functorial with respect to the isomorphism of descent data.
\end{proof}

In the notations above let $\nabla \colon \maA \to \Omega^1(M, \maA)$ be connection defined for a fixed descent datum by
\begin{equation}
(\nabla \xi)|_{U_{\alpha}}=[\nabla_{\alpha}, \xi] .
\end{equation}

It is easy to see that $\nabla$ is well defined and by derivation with respect to the product on $\maA$.

Note that
\begin{equation}
\nabla^2=[\theta, \cdot] \text{ and }\nabla \theta =2 \pi i \varOmega
\end{equation}
where $\varOmega$ is the $3$-curvature form of the connection on $\maL$, see Lemma \ref{defomega}.

\begin{proposition}\label{twchern}
\begin{enumerate}
\item Let $\maE$ be an $\maL$-twisted bundle and $\nabla$ a connection on $\maE$.
Set $\Ch_{\maL}(\nabla)= \tr e^{-\frac{u\theta}{2 \pi i}} \in \Omega^*(M)[u]$. Then $ d_{\varOmega}\Ch_{\maL}(\nabla)=0$
\item The class of $ \Ch_{\maL}(\nabla)$ in $H_{\varOmega}(M)$ is independent of choice of connection.
Namely assume we are given a different connection $\nabla'$ on $\maE$ (and therefore on $\maL$)
 and let $\varOmega'$ be the associated $3$-curvature form. Then $I([\Ch_{\maL}(\nabla)])=[\Ch_{\maL}(\nabla')]$, where $I$ is the canonical isomorphism of cohomology
of $\left( \Omega^*(M)[u], d_{\varOmega}\right)$ with cohomology
of $\left( \Omega^*(M)[u], d_{\varOmega'}\right)$.
\end{enumerate}
\end{proposition}
We  denote the class of $\Ch_{\maL}(\nabla)$ by $\Ch_{\maL}(\maE)$.
\begin{proof}

Since $\varOmega$ is central, $
\nabla( e^{-\frac{u\theta}{2 \pi i}})=-u \varOmega  e^{-\frac{u\theta}{2 \pi i}}.$
Hence
$$
d  \tr e^{-\frac{u\theta}{2 \pi i}} = \tr \nabla( e^{-\frac{u\theta}{2 \pi i}})= -u\varOmega\wedge \Ch_{\maL}(\nabla)
$$
and $ d_{\varOmega}\Ch_{\maL}(\nabla)=0$, which proves the first statement.

Fix now a descent datum $( U_{\alpha}, \maL_{\alpha\beta}, \mu_{\alpha\beta\gamma}, \maE_{\alpha}, \varphi_{\alpha \beta})$ for $\maE$ together with a connection $\nabla=(\nabla_{\alpha}, \nabla_{\alpha \beta}, \omega_{\alpha})$ and let $\nabla'=(\nabla_{\alpha}', \nabla_{\alpha \beta}', \omega_{\alpha}')$ be another connection on $\maE$.
Set $\delta_{\alpha \beta} =\nabla_{\alpha \beta}'-\nabla_{\alpha \beta}$ and let $\delta_{\alpha}$ be such that $\delta_{\alpha}-\delta_{\beta}=\delta_{\alpha \beta}$. Recall that the  isomorphism $I$ is induced by the map of complexes
\[
\xi \longmapsto e^{-u\eta}\wedge \xi
\]
where   $\eta \in \Omega^2(M)$ is defined by $\eta|_{U_{\alpha}}=\frac{1}{2\pi i}(\omega_{\alpha}'-\omega_{\alpha} -d \delta_{\alpha})$.

Consider now the manifold $\widetilde{M}=\mathbb{R} \times M$. Denote by $\pi\colon  \widetilde{M} \to  M$ the projection
on the second factor and by $t \colon \widetilde{M} \to \mathbb{R}$ the projection on the first factor.  Consider on $\widetilde{M}$ the gerbe $\widetilde{\maL}=\pi^*{\maL}$ given by the descent datum
$(\pi^{-1}U_\alpha, \pi^*\maL_{\alpha\beta}, \pi^*\mu_{\alpha\beta\gamma})$. Then $\pi^*\maE$, $\pi^*\phi_{\alpha\beta}$ describe a $\widetilde{\maL}$ twisted bundle $\widetilde{\maE}$. Moreover,
$$
\widetilde{\nabla}_{\alpha \beta} := (1-t)\pi^*\nabla_{\alpha \beta}+t\pi^*\nabla_{\alpha \beta}' \text{ and }\widetilde{\omega}_{\alpha}=(1-t)\pi^*\omega_{\alpha }+t\pi^*\omega_{\alpha }'+dt \wedge\delta_{\alpha}
$$  define respectively, connective structure and curving of $\widetilde{L}$. Similarly,  $\widetilde{\nabla}_{\alpha}=(1-t)\pi^*\nabla_{\alpha}+t\pi^*\nabla_{\alpha }'$ defines a compatible connection on $\widetilde{\maE}$. The corresponding $3$-form is given by
$$
\widetilde{\varOmega}= (1-t)\pi^*\varOmega+t\pi^*\varOmega' +\frac{1}{2 \pi i}dt \wedge (\omega_{\alpha}'-\omega_{\alpha} -d \delta_{\alpha})= (1-t)\pi^*\varOmega+t\pi^*\varOmega' + dt \wedge \pi^*\eta.
$$
Now, the map $\xi \mapsto e^{ut\eta}\wedge \xi$ is an isomorphism of complexes $\left( \Omega^*(\widetilde{M})[u], d_{\widetilde{\varOmega}}\right)$ and  $\left( \Omega^*(\widetilde{M})[u], d_{\pi^*\varOmega}\right)$. Set $\widetilde{\maA}=\pi^*\maA$ and let $\widetilde{\theta} \in \Omega^2(\widetilde{M}, \widetilde{\maA})$
be the form defined by $\pi^*\theta_{\alpha}+\widetilde{\omega}_{\alpha}$.
 By the result of the first
part of the proposition the differential form $  \tr e^{-\frac{u\widetilde{\theta}}{2 \pi i}}$ is a cocycle in the complex
$\left( \Omega^*(\widetilde{M})[u], d_{\widetilde{\varOmega}}\right)$. Hence $\tr e^{ut\eta}\wedge e^{-\frac{u\widetilde{\theta}}{2 \pi i}}$ is a cocycle in the complex $\left( \Omega^*(\widetilde{M})[u], d_{\pi^*\varOmega}\right)$.
This implies the relation
$$
\tr e^{u\eta}\wedge e^{-\frac{u\theta'}{2 \pi i}}-\tr  e^{-\frac{u\theta}{2 \pi i}}=
d_{\varOmega} \int_0^1 \iota_{\frac{\partial}{\partial t}}\tr e^{ut\eta}\wedge e^{-\frac{u\widetilde{\theta}}{2 \pi i}} dt.
$$
Therefore, we finally deduce that the differential forms $\tr e^{u\eta}\wedge e^{-\frac{u\theta'}{2 \pi i}}$ and $\tr  e^{-\frac{u\theta}{2 \pi i}}$
are cohomologous in $\left( \Omega^*(M)[u], d_{\varOmega}\right)$, and hence that the differential forms $\tr  e^{-\frac{u\theta'}{2 \pi i}}$ and $\tr  e^{-u\eta}\wedge e^{-\frac{u\theta}{2 \pi i}}$ are cohomologous in $\left( \Omega^*(M)[u], d_{\varOmega'}\right)$, which finishes the proof.
\end{proof}

We will need also a notion of superconnection on the twisted bundle. We now briefly indicate the modifications
which need to be made to the notion of connection to obtain that of superconnection.
Assume that we are given a gerbe $\maL$  and  a $\Z_2$-graded $\maL$-twisted vector bundle $\maE=\maE^+\oplus \maE^-$.
Let  $(U_\alpha, \maL_{\alpha\beta}, \mu_{\alpha\beta\gamma}, \maE_\alpha, \varphi_{\alpha\beta}) $ be a descent datum for $\maE$.
\begin{definition}
 A superconnection $\A$ on the descent datum   is $\A= (\A_{\alpha}, \nabla_{\alpha \beta}, \omega_{\alpha})$ where
$(\nabla_{\alpha \beta}, \omega_{\alpha})$ is a connection on the descent datum for $\maL$ and
each $\A_{\alpha}$ is a superconnection on  $\maE_{\alpha}$ satisfying the relations
$$
\varphi_{\alpha\beta}^* \A_\beta = \A_\alpha\otimes \id + \id \otimes \nabla_{\alpha\beta},\quad \text{ for } U_{\alpha\beta}\not = \emptyset.
$$
\end{definition}

Each superconnection $\A_{\alpha}$ can be written as $ \A_{\alpha}= \sum_{k\ge 0} \A_{\alpha}^{[k]}$ where $\A_{\alpha}^{[k]} \in \Omega^k(U_{\alpha}; \End(\maE_{\alpha})^-)$ for $k$-even, $\A_{\alpha}^{[k]} \in \Omega^k(U_{\alpha}; \End(\maE_{\alpha})^+)$ for $k$-odd, $k\ne 1$, and $\A_{\alpha}^{[1]}$ is a grading preserving connection on $\maE_{\alpha}$. It is easy to see that for each $k \ne 1$ there exists a form $\A^{[k]} \in \Omega^k(M; \End(\maE))$ such that $\A^{[k]}|{U_{\alpha}}=\A_{\alpha}^{[k]}$. For $k=1$ $(\A_{\alpha}^{[1]},  \nabla_{\alpha \beta}, \omega_{\alpha})$ defines a  connection on the (descent datum of) $\maE$.

Let now $u^{1/2}$ be a formal variable of degree $-1$ such that $(u^{1/2})^2=u$. Define the rescaled superconnection
$$
\A_{u^{-1}} := \sum u^{(k-1)/2}\A^{[k]}.
$$
Let $\varOmega$ be as before the curvature $3$-form  of the connection on $\maL$. Define the curvature of the rescaled superconnection $\theta^{\A_{u^{-1}}}$ by
$$
\theta^{\A_{u^{-1}}}|_{U_{\alpha}}=(\A_{\alpha})_{u^{-1}}^2+\omega_{\alpha}.
$$
 Then $u\theta^{\A_{u^{-1}}} \in \Omega^{even}(M, \End (\maE)^+)[u] +u^{1/2}\Omega^{odd}(M, \End (\maE)^-)[u]$. We therefore have the differential form $\exp\left(-\frac{u\theta^{\A_{u^{-1}}}}{2 \pi i}\right)$ which belongs to
$\Omega^{even}(M, \End (\maE)^+)[u] +u^{1/2}\Omega^{odd}(M, \End (\maE)^-)[u]$ and the differential form $\str  \exp \left(-\frac{u\theta^{\A_{u^{-1}}}}{2 \pi i}\right)$ which belongs to $\Omega^{even}(M)[u]$.
The following is an analogue of the Proposition \ref{twchern} for the superconnections with the essentially identical proof.

\begin{proposition}\label{chernsuperconnection}\

\begin{enumerate}
\item Set $\Ch_{\maL}(\A)= \str \exp \left(-\frac{u\theta^{\A_{u^{-1}}}}{2 \pi i}\right)$. Then $ d_{\varOmega}\Ch_{\maL}(\A)=0$
\item
The class of $ \Ch_{\maL}(\A)$ in $H_{\varOmega}(M)$ is independent of choice of superconnection. Specifically, assume we are given a different superconnection $\A'$ on $\maE$ (and therefore a different connection on $\maL$) and let $\varOmega'$ be the associated $3$-curvature form. Then $I([\Ch_{\maL}(\A)])=[\Ch_{\maL}(\A')]$, where $I$ is the canonical isomorphism of cohomology
of $\left( \Omega^*(M)[u], d_{\varOmega}\right)$ with cohomology
of $\left( \Omega^*(M)[u], d_{\varOmega'}\right)$.
\end{enumerate}
\end{proposition}

\subsection{Horizontally twisted bundles}\label{hotwbu}
Let $\pi: M \to B$ be a smooth fibration. Let $\maL$ be a gerbe on $B$.

\begin{definition}
 A descent datum
for a horizontally $\maL$-twisted bundle $\maE$ on $M$ consists of the descent datum
$(U_{\alpha}, \maL_{\alpha\beta}, \mu_{\alpha\beta\gamma})$ for $\maL$ together with a collection
$(\maE_\alpha\to \pi^{-1}U_\alpha)_{\alpha\in \Lambda}$ of vector bundles and a collection of vector bundle isomorphisms $\varphi_{\alpha\beta}\colon \maE_\alpha\otimes \pi^*\maL_{\alpha\beta} \cong \maE_\beta$  so that
$$
(\pi^{-1}U_{\alpha}, \pi^*\maL_{\alpha\beta}, \pi^*\mu_{\alpha\beta\gamma}, \maE_\alpha, \varphi_{\alpha\beta} )
$$
is a descent datum for a twisted vector bundle on $M$.
\end{definition}

An (iso)morphism between two such descent data
$$
( U_{\alpha}, \maL_{\alpha\beta}, \mu_{\alpha\beta\gamma}, \maE_{\alpha}, \varphi_{\alpha \beta}) \text{ and }
(  U_{\alpha}', \maL_{\alpha\beta}', \mu_{\alpha\beta\gamma}', \maE_{\alpha}', \varphi_{\alpha \beta}')
$$
 is given by the collection  $(\rho_{\alpha}, S_{\alpha}, \lambda_{\alpha \beta})$ where $( S_{\alpha}, \lambda_{\alpha \beta})$ is an isomorphism between
$(  U_{\alpha}, \maL_{\alpha\beta}, \mu_{\alpha\beta\gamma})$ and
$( U_{\alpha}', \maL_{\alpha\beta}', \mu_{\alpha\beta\gamma}')$  and $\rho_{\alpha} \colon \maE_{\alpha} \otimes \pi^*S_{\alpha} \to \maE_{\alpha}'$ is such that $(\rho_{\alpha}, \pi^*S_{\alpha}, \pi^*\lambda_{\alpha \beta})$ is
an (iso)morphism between  $(\pi^{-1} U_{\alpha}, \pi^*\maL_{\alpha\beta}, \pi^*\mu_{\alpha\beta\gamma}, \maE_{\alpha}, \varphi_{\alpha \beta})$ and
$( \pi^{-1} U_{\alpha}', \pi^*\maL_{\alpha\beta}', \pi^*\mu_{\alpha\beta\gamma}', \maE_{\alpha}', \varphi_{\alpha \beta}')$.

With these definitions one can now define a horizontally twisted bundle as an equivalence class of descent data.
Let $Tw(\pi^*\maL)$ denote the set of isomorphism classes of
all $\pi^*\maL$-twisted bundles on $M$ and $Tw_h(\maL)$ denote the set of isomorphism classes of
all horizontally $\maL$-twisted bundles. Then we have an obvious map $Tw_h(\maL) \to Tw(\pi^*\maL)$. According to Lemma \ref{extend} this map is surjective. In particular $Tw_h(\maL) \ne \emptyset$ if and only if $\pi^* [\maL]$ is torsion in $H^3(M, \Z)$. It is however not injective. Indeed, if $(U_{\alpha}, \maL_{\alpha\beta}, \mu_{\alpha\beta\gamma}, \maE_\alpha, \varphi_{\alpha \beta})$ is a descent datum for a horizontally twisted bundle and
$S$ is a line bundle on $M$ then $(U_{\alpha}, \maL_{\alpha\beta}, \mu_{\alpha\beta\gamma}, \maE_\alpha \otimes S|_{\pi^{-1}U_\alpha}, \varphi_{\alpha \beta}\otimes \id)$ is another such descent datum. These data define the same
element of $Tw(\pi^*\maL)$ but, unless $S$ is a pull-back of a line bundle from $B$, different elements of $Tw_h(\maL)$.

A connection on the descent datum $( U_{\alpha}, \maL_{\alpha\beta}, \mu_{\alpha\beta\gamma}, \maE_{\alpha}, \varphi_{\alpha \beta})$ is a collection $(\nabla_{\alpha}, \nabla_{\alpha \beta}, \omega_{\alpha})$ where
$(\nabla_{\alpha \beta}, \omega_{\alpha})$ is a connection on the descent datum for $\maL$ and $\nabla_{\alpha}$ is a
connection on $\maE_{\alpha}$ such that $(\nabla_{\alpha}, \pi^*\nabla_{\alpha \beta}, \pi^*\omega_{\alpha})$
is a connection on the descent datum  $(\pi^{-1} U_{\alpha}, \pi^*\maL_{\alpha\beta}, \pi^*\mu_{\alpha\beta\gamma}, \maE_{\alpha}, \varphi_{\alpha \beta})$. 
With these
definitions one can now define a notion of connection on the horizontally $\maL$-twisted bundle in complete
analogy with the defnitions for the twisted bundles. If $\nabla$  is such a connection and $\varOmega$ is the curvature $3$-form of the gerbe $\maL$, one defines $\Ch_{\maL}(\nabla)$ -- a closed form in $\left(\Omega^*(M), d_{\pi^*\varOmega}\right)_{\bullet}$.
The analogues of Propositions \ref{global curvature} and \ref{twchern} hold in this context with the same proofs.

\section{Projective families and the analytic index}\label{projective}

\subsection{Families of pseudodifferential operators}
Here we collect several facts about the (untwisted) families of pseudodifferential operators.

%
%
%
%
Let $\pi \colon X \to Y$ be a smooth fibration and $E$ a vector bundle on $X$.

We denote by $\Psi_\maL^m (X|Y ; E)$ the space of classical  fiberwise pseudodifferential operators of order $\leq m$ on   $\pi$, acting on the sections of  the vector bundle $E$. As usual, we set
$$
\Psi (X|Y ; E) := \bigcup_{m\in \Z} \Psi^m (X|Y ; E)
$$
 and
$$
 \Psi^{-\infty} (X|Y; E) := \bigcap_{m\in \Z} \Psi^m (X|Y ; E).
$$

Recall that composition endows each $\Psi(X|Y ; E)$ with the structure of a filtered algebra and that $\Psi^{-\infty} (X|Y; E)$ is an ideal in this algebra. $\Psi(X|Y ; E)$ is also a module over $C^{\infty}(Y)$ and the composition is
$C^{\infty}(Y)$-linear.

We have the following elementary general result:
\begin{lemma} Assume $\pi\colon X\to Y$ is a smooth fibration, $E$ a vector bundle on $X$, $L$ a line bundle on $Y$.
Define a map $\chi_L\colon \Psi  (X|Y ; E) \to \Psi (X|Y ; E\otimes \pi^* L )$ by
\[
\chi_L(D) (e\otimes \pi^* (l)) = D(e)\otimes \pi^*l
\]
for $D\in \Psi_\maL (X|Y ; \maE)$, $e \in \Gamma_c(E)$, $l \in \Gamma(L)$. Then $\chi_L$ is a well-defined  isomorphism of  algebras and $C^{\infty}(Y)$-modules.
\[
\chi_{L_1 \otimes L_2} = \chi_{L_2} \circ \chi_{L_1}
\]
\end{lemma}

If we have two vector bundles $E$, $E'$ on $X$ we denote by $\Psi(X|Y; E, E')$ the set of fiberwise pseudodifferential
operators $\Gamma_c(E) \to \Gamma(E')$. For $L$ -- line bundle on $Y$ we again  have the isomorphism of $C^{\infty}(Y)$-modules $\chi_L \colon   \Psi(X|Y; E, E') \to \Psi(X|Y; E\otimes \pi^*L, E'\otimes \pi^*L)$ defined
by the same formula.

We have the vertical cotangent bundle $T^*(X|Y)= T^*X/(Ker \pi_*)^{\perp}$. $\mathring{T}^*(X|Y)$ denotes
(the total space of) this bundle with the zero section removed, and $p \colon \mathring{T}^*(X|Y) \to X$ is the natural projection.
Recall that for $P\in \Psi^m (X|Y ; E, E')$    the principal symbol $\sigma_m(P)$ is an $m$-homogeneous smooth section over  $\mathring{T}^*(X|Y) $  of vector bundle $p^*\Hom( E, E')$. Then identifying canonically isomorphic bundles $\Hom( E, E')$ and  $\Hom( E\otimes \pi^*L, E'\otimes \pi^*L)$ we have $\sigma_m(P) =\sigma_m(\chi_L(P))$.

\subsection{Projective families}\label{prfa}

Let $\pi: M \to B$ be a smooth fibration with compact fibers. Let $\maL$ be a gerbe on $B$ such that $\pi^*[\maL]$ is a torsion class in $H^3(M, \Z)$. Let $\maE$ be a horizontally $\maL$-twisted bundle  on $M$, cf. Section \ref{hotwbu}.
We fix a descent datum $(U_{\alpha}, \maL_{\alpha\beta}, \mu_{\alpha\beta\gamma}, \maE_\alpha, \varphi_{\alpha \beta})$ for $\maE$.
  For any $(\alpha, \beta)\in \Lambda^2$ with $U_{\alpha\beta}\not = \emptyset$
we have an isomorphism of filtered  algebras, respecting the $C^{\infty}(U_{\alpha \beta})$-module structure:
\begin{equation}\label{phi}
 \phi_{\alpha\beta}\colon \Psi (\pi^{-1}U_{\alpha\beta}|U_{\alpha\beta} ; \maE_\beta)\rightarrow \Psi (\pi^{-1}U_{\alpha\beta}|U_{\alpha\beta} ; \maE_\alpha),
\end{equation}

It is defined as the composition
\[
\Psi(\pi^{-1}U_{\alpha\beta}|U_{\alpha\beta} ; \maE_\beta) \overset{\psi_{\alpha \beta}}\to 
\Psi (\pi^{-1}U_{\alpha\beta}|U_{\alpha\beta} ; \maE_\beta \otimes \pi^* \maL_{\alpha \beta}) \overset{\varphi_{\alpha \beta}}\to \Psi (\pi^{-1}U_{\alpha\beta}|U_{\alpha\beta} ; \maE_\alpha)
\]
where $\psi_{\alpha \beta} =\chi_{\maL_{\alpha \beta}}^{-1}=\chi_{\maL_{\beta \alpha }}$.

Recall (cf. \cite{bgv}) that  for every $\alpha \in \Lambda$ we have an infinite dimensional bundle $\pi_* \maE_{\alpha}$
on $U_{\alpha}$ defined by $\Gamma(V, \pi_*\maE_{\alpha}) = \Gamma(\pi^{-1}V, \maE_{\alpha})$, $V \subset U_{\alpha}$.
Over $U_{\alpha \beta}$ we have isomorphisms $\pi_*\varphi_{\alpha \beta} \colon \pi_*\maE_{\alpha} \otimes \maL_{\alpha \beta} \to \pi_* \maE_{\beta}$ defined by
\[
\pi_*\varphi_{\alpha \beta} (\xi \otimes l) = \varphi (\xi \otimes \pi^*(l)).
\]
Here $\xi \in \Gamma(U_{\alpha \beta}, \pi_*\maE_{\alpha})= \Gamma(\pi^{-1}U_{\alpha \beta}, \maE_{\alpha})$, $l \in \Gamma(U_{\alpha \beta}, \maL_{\alpha \beta})$.

Note that the isomorphisms $\pi_*\varphi_{\alpha \beta} \colon \pi_*\maE_{\alpha} \otimes \maL_{\alpha \beta} \to \pi_* \maE_{\beta}$ induce the isomorphisms
\[(\pi_*\varphi_{\alpha \beta})^* \colon \End(\pi_* \maE_{\beta})\to
\End(\pi_*\maE_{\alpha} \otimes \maL_{\alpha \beta})\cong \End(\pi_*\maE_{\alpha})\]
over $U_{\alpha \beta}$. The restriction of this isomorphism to $\Psi (\pi^{-1}U_{\alpha \beta}|U_{\alpha \beta}, \maE_\beta) \subset \End(\pi_* \maE_{\beta})$ coincides with the isomorphism $\phi_{\alpha \beta} \colon \Psi (\pi^{-1}U_{\alpha \beta}|U_{\alpha \beta}, \maE_\beta) \to \Psi (\pi^{-1}U_{\alpha \beta}|U_{\alpha \beta}, \maE_\alpha)$. Since the isomorphisms $(\pi_*\varphi_{\alpha \beta})^*$ satisfy the natural cocycle identity
we have the following:

\begin{lemma}\label{Comp.Local}
The isomorphisms $\phi_{\alpha \beta}$ satisfy
\[\phi_{\alpha\beta}\circ \phi_{\beta\gamma} = \phi_{\alpha\gamma}
\]
whenever $U_{\alpha\beta\gamma}\not = \emptyset$.
\end{lemma}

Recall that the isomorphisms $\varphi_{\alpha \beta}$ induce the natural isomorphisms $\varphi_{\alpha \beta}^* \colon \End(\maE_{\beta}) \to \End(\maE_{\alpha})$. Then
\begin{equation}\label{symbolinv}
\sigma_m \circ \phi_{\alpha \beta} = p^*\left(\varphi_{\alpha \beta}^*\right) \circ \sigma_m
\end{equation}
\begin{definition}
A fiberwise pseudodifferential operator $P$ of order $\leq m$ with coefficients in the horizontally $\maL$-twisted vector bundle $\maE$ is a collection $\{P_\alpha\}_{\alpha\in \Lambda}$,  $P_\alpha \in \Psi^m (\pi^{-1}U_\alpha|U_\alpha ; \maE_\alpha)$ such that
$$
P_\alpha = \phi_{\alpha\beta} (P_\beta).
$$
where $\phi_{\alpha\beta}$ is defined in Equation \ref{phi}.
The space of fiberwise pseudodifferential operators of order $\leq m$, with coefficients in the $\pi^*\maL$-twisted vector bundle $\maE$, is denoted by $\Psi_\maL^m (M|B ; \maE)$.
\end{definition}
Note that the equation \eqref{symbolinv} implies that if $P=\{P_{\alpha}\} \in \Psi_\maL^m (M|B ; \maE)$ then
the collection $\sigma_m(P_{\alpha})$ defines a section of the (untwisted) bundle $p^*\End(\maE)$. We will call this
section the principal symbol of $P= \{P_{\alpha}\}$

\begin{remark}
We define in the same way the space $\Psi_\maL^m (M|B ; \maE, \maE')$ of fiberwise pseudodifferential operators of order $\leq m$, from the horizontally $\maL$-twisted vector bundle $\maE$ to the horizontally $\maL$-twisted vector bundle $\maE'$. In particular $\Psi_\maL^m (M|B ; \maE, \maE) = \Psi_\maL^m (M|B ; \maE)$. We also have a principal symbol map $\sigma_m \colon \Psi_\maL^m (M|B ; \maE, \maE) \to p^* \Hom( \maE, \maE')$.
\end{remark}

We set
$$
\Psi_\maL (M|B ; \maE):= \bigcup_{m\in \Z} \Psi_\maL^m (M|B ; \maE) \text{ and } \Psi_\maL^{-\infty} (M|B ; \maE):= \bigcap_{m\in \Z} \Psi_\maL^m (M|B ; \maE).
$$

Introduce now a composition in $\Psi_\maL (M|B ; \maE)$ by

\[
\{P_\alpha\} \circ \{Q_{\alpha}\} = \{P_{\alpha} Q_{\alpha}\}
\]
Since $\phi_{\alpha \beta}$ are algebra isomorphisms the right hand side of this equality defines an element
in $\Psi_\maL (M|B ; \maE)$.
\begin{proposition}\label{Comp}
The composition of operators is $C^{\infty}(B)$-linear and endows $\Psi_\maL (M|B ; \maE)$  with the structure of associative   algebra; $\Psi_\maL^{-\infty} (M|B ; \maE)$ is an ideal in $\Psi_\maL (M|B ; \maE)$.
\end{proposition}

We can now define the algebra of forms on $B$ with values in $\Psi_\maL(M|B; \maE)$ by
\[
\Omega^*\left(B, \Psi_\maL(M|B; \maE)\right)= \Omega^*(B) \otimes_{C^{\infty}(B)} \Psi_\maL(M|B; \maE).
\]
Recall that for every $\alpha$ and $V\subset U_{\alpha}$ we have a fiberwise trace $\Tr_{\alpha} \colon \Psi^{-\infty}(\pi^{-1}V|V; \maE_{\alpha}) \to C^{\infty}(V)$. It is easy to see that for $ P \in
\Psi^{-\infty}(\pi^{-1}U_{\alpha \beta}|U_{\alpha \beta}; \maE_{\beta})$
\[
\Tr_{\alpha} \phi_{\alpha \beta} (P) = \Tr_{\beta} (P).
\]

We therefore obtain a well defined map $\Tr \colon \Psi_\maL^{-\infty}(M|B, \maE) \to C^{\infty}(B)$
by setting
$\left.\Tr \{P_{\alpha}\}\right|_{U_{\alpha}}= \Tr_{\alpha}(P_{\alpha})$.
This trace is a $C^{\infty}(B)$-module map satisfying $\Tr[A, B]=0$.
It extends naturally to define a map $\Tr \colon \Omega^*\left(B, \Psi_\maL^{-\infty}(M|B; \maE)\right) \to
\Omega^*\left(B\right)$.

If the bundle $\maE$ is $\Z_2$ graded we have a similarly defined supertrace
\[\STr \colon \Omega^*\left(B, \Psi_\maL^{-\infty}(M|B; \maE)\right) \to
\Omega^*\left(B\right).
\]

Note that the definition of the $\Psi_\maL (M|B ; \maE)$ depends on the descent datum for $\maE$. It is straightforward however to see that an isomorphism of descent data defines canonically an isomorphism of the corresponding
bundles of algebras.

\subsection{Analytic index}\label{analytic}
\begin{definition}
Let $m\geq 0$ be fixed. A fiberwise pseudodifferential operator $P\in \Psi_\maL^m (M|B ; \maE, \maE')$ is fiberwise elliptic if the principal symbol $\sigma_m(P) \in \Gamma(p^*\Hom(\maE, \maE'))$is an isomorphism.
\end{definition}

We say that $Q \in \Psi_\maL^{-m} (M|B ; \maE', \maE)$ is a parametrix of $P\in \Psi_\maL^m (M|B ; \maE, \maE')$ if
$PQ-1 \in \Psi_\maL^{-\infty} (M|B ; \maE', \maE')$ and $QP-1 \in \Psi_\maL^{-\infty} (M|B ; \maE, \maE)$.

\begin{lemma} Every elliptic $P\in \Psi_\maL^m (M|B ; \maE, \maE')$ has a parametrix.
\end{lemma}
\begin{proof}
Construct first a parametrix $R_{\alpha}$ for $P_{\alpha}$. Let $\rho_{\alpha}$ be a partition of
unity subordinate to $\{U_{\alpha}\}$. Then define $R_{\alpha}' = \sum_{\beta} \rho_{\beta} \phi_{\alpha\beta} (R_{\beta})$. Then each $R_{\alpha}'$ is a parametrix for $P_{\alpha}$ and $\phi_{\alpha\beta}(R_{\beta}')=R_{\alpha}'$.
\end{proof}

Let $D\in \Psi_\maL^m (M|B ; \maE, \maE')$ be elliptic. Let $F\in \Psi_\maL^0 (M|B ; \maE, \maE')$ be such that $\sigma_0(F)|_{S^*(M|B)}=\sigma_m(D)|_{S^*(M|B)}$.
 Choose a parametrix
$R$ for $F$.
 Let $U_{D} \in \Psi_\maL^{0} (M|B ; \maE\oplus \maE')$ be an invertible operator such that $U_{D}-\begin{bmatrix} 0 &-R\\
F &0 \end{bmatrix} \in  \Psi_\maL^{-\infty} (M|B ; \maE\oplus \maE')$. An explicit construction of an example of such an
operator is as follows.  Let $S_0=1-RF$, $S_1=1-FR$. Then set $U_{D}=\begin{bmatrix} S_0 &-(1+S_0)R\\F &S_1 \end{bmatrix}$. With such a choice the inverse is given by an explicit formula  $U_{D}^{-1}=\begin{bmatrix} S_0 &(1+S_0)R\\-F &S_1 \end{bmatrix}$.

\begin{definition}\label{defindex}
The index of $D$ is the $K$-theory class of the algebra $\Psi_\maL^{-\infty}(M|B; \maE\oplus \maE')$ defined by
$$
\ind(D) = \left[P_{D} - Q \right] \in K_0(\Psi_\maL^{0} (M|B, \maE\oplus \maE'), \Psi_\maL^{-\infty} (M|B, \maE\oplus \maE')) \cong K_0(\Psi_\maL^{-\infty} (M|B, \maE\oplus \maE')),
$$
where $P_{D}$ and $Q$  are the idempotents given by $P_{D}= U_{D}  \begin{bmatrix} 1 &0 \\0 &0\end{bmatrix} U_{D}^{-1}$ and $Q= \begin{bmatrix}0 &0\\0 &1 \end{bmatrix}$.
\end{definition}


We leave to the reader the standard $K$-theoretic proof that the index is well defined and is stable under the
homotopies of $D$ in the class of elliptic operators in $\Psi_\maL^m (M|B ; \maE, \maE')$.
Assume that the horizontally $\maL$-twisted bundles $\maE$ and $\maE'$ are hermitian (and in particular $\maL$ is unitary) and that fibers of $\pi$ are equipped with smoothly varying volume forms. In this situation for a projective family $D$ we can define an formally adjoint projective family $D^*$ by forming the formal adjoints for each family $D_{\alpha}$.

\begin{lemma}  Then, identifying $K_0(\Psi_\maL^{-\infty} (M|B, \maE\oplus \maE'))$ with $K_0(\Psi_\maL^{-\infty} (M|B, \maE'\oplus \maE))$ we have
\[\ind D^*=-\ind D.\]
\end{lemma}
\begin{proof}
We have
$$
\ind D^* = \left[U_{D^*}  \begin{bmatrix} 0 &0 \\0 &1\end{bmatrix} U_{D^*}^{-1} - \begin{bmatrix}1 &0\\0 &0 \end{bmatrix} \right] \in K_0(\Psi_\maL^{-\infty} (M|B, \maE\oplus \maE')).
$$
By deforming $D$ we may assume that $\sigma_m(D)|_{S^*(M|B)}$ is an isometry. In this case we may choose $U_{D^*}=U_D^{-1}$, and the statement follows.
\end{proof}


\subsection{Chern character of the index}
We continue in the notations of the previous section. We assume that we are given a horizontally $\maL$-twisted
bundle $\maE$ with a connection represented by a descent datum $(U_{\alpha}, \maL_{\alpha \beta}, \mu_{\alpha \beta \gamma}, \maE_{\alpha}, \varphi_{\alpha \beta})$ with connection $(\nabla_\alpha, \nabla_{\alpha \beta}, \omega_{\alpha})$   see the Section \ref{hotwbu}.
Following Mathai and Stevenson \cite{MathaiStevenson}, we describe  in this paragraph a morphism of complexes
$
CC_{\bullet}^-(\Psi_\maL^{-\infty} (M|B;\maE)) \to \left(\Omega^*(M)[u], d_{\varOmega} \right)_{\bullet}.
$

Recall the bundles $\pi_*\maE_{\alpha}$ and isomorphisms $\pi_*\varphi_{\alpha \beta}$ defined in the Section \ref{prfa}. It is easy to see that
$(U_{\alpha}, \maL_{\alpha \beta}, \mu_{\alpha \beta \gamma}, \pi_*\maE_{\alpha}, \pi_* \varphi_{\alpha \beta})$
is a descent datum for an infinite dimensional twisted bundle.
We now proceed to define a connection on this descent datum.

Choose  horizontal distribution i.e. a subbundle $\maH \subset TM$ such that
$
TM = \maH \oplus T (M|B).
$
This choice together with connections $\nabla^{\maE}_{\alpha}$ defines for each $\alpha$ a connection $\nabla_{\alpha}^{\maH}$ as follows:
\[
(\nabla^{\maH}_{\alpha})_X\xi = (\nabla^{\maE}_{\alpha})_{X^{\maH}} \xi
\]
where $X^{\maH}$ is the horizontal lift of $X\in \Gamma(B, TB)$.

\begin{lemma}\label{pi*}
$(\pi_*\varphi_{\alpha \beta})^* \nabla^{\maH}_{\beta} = \nabla^{\maH}_{\alpha} \otimes \id+ \id \otimes \nabla_{\alpha \beta}$
\end{lemma}

The curvature of the connection $\nabla^{\maH}_{\alpha}$ is a $2$-form $\theta^{\maH}_{\alpha}$ on $U_{\alpha}$
with values in fiberwise differential operators  given by
\[
\theta^{\maH}_{\alpha}(X, Y)=\theta^{\maE}_{\alpha}(X^{\maH}, Y^{\maH})+ (\nabla^{\maE}_{\alpha})_{T(X, Y)}.
\]
where
\begin{equation}\label{th}
T^{\maH}(X, Y) = [X^{\maH}, Y^{\maH}]-[X, Y]^{\maH}, \ X, Y \in \Gamma(B, TB).
\end{equation}
Each $\nabla^{\maH}_{\alpha}$ defines a filtration-preserving derivation $\partial^{\maH}_{\alpha}$  of the algebra of fiberwise pseudodifferential operators
\[
\partial_{\alpha}^{\maH} \colon \Psi(\pi^{-1}U_{\alpha}|U_{\alpha}, \maE_{\alpha}) \rightarrow
\Omega^1(U_{\alpha}, \Psi(\pi^{-1}U_{\alpha}|U_{\alpha}, \maE_{\alpha})) \text{ defined  by }
\partial^{\maH}_{\alpha} (D) = [\nabla_{\alpha}^{\maH}, D].
\]
If $D \in \Psi(\pi^{-1}U_{\alpha \beta}|U_{\alpha \beta}, \maE_{\beta})$ then the result of Lemma \ref{pi*}
implies that  \[\partial^\maH_{\alpha} (\phi_{\alpha \beta} (D))= \phi_{\alpha \beta} (\partial^\maH_{\beta}(D)).\]
Therefore if $\{D_{\alpha}\}$, $D_{\alpha} \in \Psi_\maL(\pi^{-1}U_{\alpha}|U_{\alpha}, \maE_{\alpha})$,
defines an element in $\Psi_\maL^m(M|B, \maE)$ then   $\{\partial_{\alpha}^{\maH}(D_{\alpha})\} \in \Omega^1(B, \Psi_\maL^m(M|B, \maE))$.
We therefore obtain a  derivation
$$
\partial^{\maH} \colon \Psi_\maL(M|B, \maE) \to  \Omega^1(B, \Psi_\maL(M|B, \maE)),
$$
which extends to a derivation of the algebra $\Omega^*(B, \Psi_\maL(M|B, \maE))$.

\begin{lemma} There exists $\theta^{\maH} \in \Omega^2(B, \Psi_\maL^1(M|B, \maE))$ such that
\[
\theta^{\maH}|_{U_{\alpha}}= \theta^{\maH}_{\alpha} +\pi^*\omega_{\alpha}
\]
\end{lemma}
\begin{proof}
By Lemma \ref{pi*}, $\phi_{\alpha \beta}^* \theta^{\maH}_{\beta} = \theta^{\maH}_{\alpha} +\pi^*(\omega_{\alpha} -\omega_\beta)$, and the statement follows as in Proposition \ref{global curvature}.
\end{proof}
We have
$$
(\partial^{\maH})^2 (D) = [\theta^{\maH}, D]\text{ and }
\partial^{\maH}(\theta^{\maH})=2\pi i(\pi^* \varOmega).
$$
where $\varOmega$ is the $3$-curvature form of the connection on $\maL$.

Following Mathai and Stevenson \cite{MathaiStevenson} one can construct the  morphism of complexes
\[\Phi_{\nabla^{\maH}} \colon CC_{\bullet}^{-}(\Psi_\maL^{-\infty} (M|B; \maE\oplus \maE')) \to \left( \Omega^*(B)[u], d_{\varOmega} \right)_{\bullet}\]
 as follows.
(Note that in the nontwisted case similar morphism was constructed
in \cite{Gorokhovsky1, NestTsygan1, NestTsygan2}.)
Denote by
$$
\Delta^k := \{(t_0, \cdots , t_{k})\in \R^{k}\ | \  0\leq t_i,  \sum_{i=0}^k t_i=1\}.
$$
the standard $k$-simplex.

Define the maps $\Phi^k_{\nabla^\maH} \colon C_k\left({\Psi_\maL}^{-\infty} (M|B; \maE\oplus \maE')\right) \to \Omega^*(M)[u]$ by
$$
\Phi^k_{\nabla^\maH}(A_0, \cdots, A_k) := \int_{\Delta^k} \Tr\left( A_0 e^{-ut_0\frac{\theta^{\maH}}{2 \pi i}} \pa^\maH (A_1) e^{-ut_1\frac{\theta^{\maH}}{2 \pi i}} \cdots e^{-ut_{k-1}\frac{\theta^{\maH}}{2 \pi i}} \pa^\maH (A_k)e^{-ut_k\frac{\theta^{\maH}}{2 \pi i}}\right) dt_1\ldots dt_k,
$$
for  $A_0\otimes\ldots \otimes A_k \in C_k\left(\Psi_\maL^{-\infty} (M|B; \maE\oplus \maE')\right)$. Then
let $\Phi _{\nabla^\maH}= \sum_{k=0}^{\infty} \Phi^k_{\nabla^\maH}$.

\begin{theorem}\cite{MathaiStevenson}
 The map  $\Phi _{\nabla^\maH} \colon (CC^{-}_{\bullet}(\Psi_\maL^{-\infty} (M|B; \maE)), b+uB) \to (\Omega^*(B)[u], d_{\varOmega})_{\bullet}$
is a morphism of complexes.
\end{theorem}

This morphism depends on the choice of horizontal distribution $\maH$. However the results of \cite{MathaiStevenson} show that a different choice of $\maH$ leads to a chain homotopic morphism.

Assume now that $D \in \Psi_{\maL}(M|B; \maE, \maE')$ is an twisted elliptic family.
In the Definition \ref{defindex} we defined $\ind(D) = \left[P_{D} - Q \right] \in K_0(\Psi_{\maL}^{-\infty}(M|B; \maE, \maE'))$. Here the idempotents $P_D$ and $Q$ belong to the algebra $\overline{\Psi_{\maL}^{-\infty}(M|B; \maE\oplus \maE'))}$ -- the unitalization of $\Psi_{\maL}^{-\infty}(M|B; \maE \oplus \maE')$ with the multiplier $Q \in \Psi_{\maL}^{0}(M|B; \maE \oplus \maE')$ adjoined. It follows that if we directly apply the formula \eqref{cyclicchern} to $\left[P_{D} - Q \right] $
we obtain a $0$-chain $\Ch \left[P_{D} - Q \right]$ in the relative cyclic complex $CC_{\bullet}^-\left( \overline{\Psi_{\maL}^{-\infty}(M|B; \maE\oplus \maE'))},  \Psi_{\maL}^{-\infty}(M|B; \maE\oplus \maE'))\right)$. We have  the natural morphism $\iota \colon CC_{\bullet}^-\left( \Psi_{\maL}^{-\infty}(M|B; \maE\oplus \maE'))\right) \to
CC_{\bullet}^-\left( \overline{\Psi_{\maL}^{-\infty}(M|B; \maE\oplus \maE'))},  \Psi_{\maL}^{-\infty}(M|B; \maE\oplus \maE'))\right)$, and the following equality of homology classes: $\iota (\ch (\ind D)) = \left[\Ch \left[P_{D} - Q \right] \right]$. It is straightforward to see that the map $\Phi _{\nabla^\maH}$ extends to a morphism
\[
\Phi _{\nabla^\maH} \colon CC_{\bullet}^-\left( \overline{\Psi_{\maL}^{-\infty}(M|B; \maE\oplus \maE'))},  \Psi_{\maL}^{-\infty}(M|B; \maE\oplus \maE'))\right) \to (\Omega^*(B)[u], d_{\varOmega})_{\bullet}
\]
defined by the same formula, and that we have an equality $\Phi _{\nabla^\maH} \circ \iota = \Phi _{\nabla^\maH}$. It follows that
\[
\Phi _{\nabla^\maH}(\ch (\ind D)) = \left[\Phi _{\nabla^\maH}(\Ch \left[P_{D} - Q \right] ) \right].
\]

\section{Dirac operators and superconnections}\label{Dirac}

\subsection{Dirac operators}\label{Dirac operators}
The goal of this section is to give a superconnection proof of the family index theorem for  projective families of  Dirac operators. We assume that the fibers of the smooth  fibration $\pi: M\to B$ are even dimensional compact Riemannian manifolds.

We begin with the definition of a horizontally twisted Clifford module.
 Denote by $C(M|B)$ the Clifford algebra of the fiberwise cotangent bundle $T^*(M|B)=T^*M/ (\ker \pi_*)^{\perp}$.
Let $\maL$ be a unitary gerbe on $B$.
\begin{definition}\
\begin{itemize}
 \item  A twisted Clifford module is a horizontally  $\maL$-twisted Hermitian $\Z_2$-graded vector bundle $\maE=\maE^+ \oplus \maE^-$ on $M$ together with the homomorphism $c \colon  C(M|B) \rightarrow \End (\maE)$ of bundles of unital $\Z_2$-graded $*$-algebras.
\item A Clifford  connection $\nabla^{\maE}$ on $\maE$ is an Hermitian  connection such that $\nabla^{\maE} (c(a))=c(\nabla^{M|B} (a))$.
\end{itemize}

\end{definition}
Clifford connections on Clifford modules always exist; the proof  is analogous to the
proof of existence of twisted connections in Lemma \ref{connection existence}.

Choose  horizontal distribution i.e. a subbundle $\maH \subset TM$ such that
$
TM = \maH \oplus T (M|B).
$
This choice together with the Riemannian metric on the fibers of $\pi$ allows one to define a connection $\nabla^{M|B}$ on the fiberwise tangent bundle $T(M|B)$, see \cite{bgv} Section 10.1.
We denote by $R^{M|B}$ the curvature of this connection.

 Set $\End_{C(M|B)}(\maE)= \{A \in \End(\maE)\ | \ [A, c(a)]=0 \text{ for every } a \in C(M|B) \}$.
Let $\Gamma \in C(M|B)$ be the chirality operator defined locally by $\Gamma = i^{k / 2} e^1\ldots e^k$ where $k=\dim M-\dim B$ and $e^1, \ldots, e^k$ is the local orthonormal basis of $T^*(M|B)$. Define then the relative supertrace
$$
\str_{\maE/\maS} \colon  \End_{C(M|B)}(\maE) \to C^{\infty} (M)\text{ by }\str_{\maE/\maS}(A)= 2^{-k/2} \str c(\Gamma) A.
$$
We fix from now on a Clifford connection $\nabla^{\maE}$ on $\maE$ and a descent datum $(U_{\alpha}, \maL_{\alpha \beta},
\mu_{\alpha \beta \gamma}, \maE_{\alpha}, \varphi_{\alpha \beta})$ for the horizontally $\maL$-twisted Clifford
module $\maE$. The connection $\nabla^{\maE}$ defines a connection $(\nabla^{\maE}_{\alpha}, \nabla_{\alpha \beta}, \omega_{\alpha})$ (defined up to equivalence) on this descent datum. Each $\maE_{\alpha}$ is then a Clifford module on the fibration $\pi^{-1}U_{\alpha} \to U_{\alpha}$, and
each connection $\nabla^{\maE}_{\alpha}$ is a Clifford connection.

Recall  (see Proposition \ref{global curvature}) that one defines $\theta^{\maE} \in \Omega^2(M, \End(\maE))$
by setting $\theta^{\maE}|_{\pi^{-1}U_{\alpha}}= \theta^{\maE}_{\alpha} +\pi^* \omega_{\alpha}$.
Denote by  $c(R^{M|B})$  the action of the $2$ -form with values in the
Clifford algebra obtained from $R^{M|B}$ via the Lie algebra isomorphism $\mathfrak{so}(T(M|B)) \to C^2(M|B)$.
Here $C^2(M|B) \subset C(M|B)$ is a subspace consisting of elements $\sum u_i v_i$, $u_i, v_i \in T^*(M|B)$ with $\sum \langle u_i, v_i \rangle=0$.
Define $\theta^{\maE/\maS}=\theta^{\maE}-c(R^{M|B})$.

The argument  in \cite{bgv}, Proposition 3.43,
shows that $\theta^{\maE/\maS} \in \Omega^2(M, \End_{C(M|B)}(\maE)) $. Over an open set $\pi^{-1}U_{\alpha}$ we have
$\theta^{\maE/\maS}|_{U_{\alpha}}=\theta^{\maE/\maS}_{\alpha} +\pi^*\omega_{\alpha}$ where $\theta^{\maE/\maS}_{\alpha} \in \End_{C(M|B)}(\maE_{\alpha})$ is defined via the equality $\theta^{\maE}_{\alpha}=\theta^{\maE/\maS}_{\alpha}+ c(R^{M|B})$. We can then define a differential form $\Ch_{\maL}(\maE/\maS)$ by
\begin{equation}\label{deftwistedchern}
\Ch_{\maL}(\maE/\maS) =\str_{\maE/\maS} e^{-\frac{u \theta^{\maE/\maS}}{2 \pi i}}  \in \Omega^*(M)[u].
\end{equation}
The  proof of the following result  is standard and analogous to the proof of  Proposition \ref{twchern}.
\begin{lemma}
We have $ d_{\pi^*\varOmega}\, \Ch_{\maL}(\maE/\maS) = 0$ and the corresponding class is $H^*_{\pi^*\varOmega}(M)$ is independent of the
choice of Clifford connection $\nabla^{\maE}$.
\end{lemma}
We also introduce the fiberwise $\widehat{A}$-genus by $\widehat{A}(TM|B)= \widehat{A}\left(\frac{u}{2 \pi i}R^{M|B}\right)\in \Omega^*(M)[u]$, where $ \widehat{A}(x)$ is the power series defined by
\[
\widehat{A}(x)=\left.\det\right.^{1/2}\left(\frac{x/2}{\sinh x/2}\right).
\]

Using the above data,  we can define on each  fibration $\pi^{-1}U_{\alpha} \to U_{\alpha}$ a family of Dirac operators $D_{\alpha}$ acting on the sections of the bundle $\maE_{\alpha}$. Locally $D_{\alpha} = \sum_i c(e^i) \left(\nabla^{\maE}_{\alpha}\right)_{e_i}$ where $\{e_i\}$, $\{e^i\}$ are dual bases of $T(M|B)$ and $T^*(M|B)$ respectively. {We leave it to the reader to check the following.

\begin{lemma} The collection $D=\{ D_{\alpha}\}$ defines an element in $\Psi_\maL^1(M|B; \maE)$.
\end{lemma}
With respect to the decomposition $\maE= \maE^+\oplus\maE^-$ we have the decomposition
$\Psi_\maL(M|B; \maE) = \Psi_\maL(M|B; \maE^+, \maE^+)\oplus\Psi_\maL(M|B; \maE^+, \maE^-)\oplus\Psi_\maL(M|B; \maE^-, \maE^+)\oplus\Psi_\maL(M|B; \maE^-, \maE^-)$. The Dirac operator then decomposes as $D= D^+\oplus D^-$ where
$D^+ \in \Psi_\maL^1(M|B; \maE^+, \maE^-)$, $D^- \in \Psi_\maL^1(M|B; \maE^-, \maE^+)$

Classical arguments show that $D^+$ is fiberwise elliptic and hence the analytical index $\ind (D^+)$ of $D^+$ is well defined in $K(\Psi_\maL^{-\infty}(M|B; \maE))$.

We then have the following result
\begin{theorem}\ The following formula holds in $H^{\bullet}_{\maL}(B)$:
\[\Phi_{\nabla^{\maH}}( \ch (\ind D^+))]
 =  u^{-\frac{\dim M -\dim B}{2}}\left[\int_{M/B} \widehat{A}\left(\frac{u}{2 \pi i}R^{M|B}\right) \wedge \Ch_{\maL}(\maE/\maS)\right].
 \]
\end{theorem}
This index theorem is established in \cite{MMS1, MMS2} for general projective families however with the more restrictive
conditions on the class $[\maL]$ of the gerbe $\maL$.
The superconnection proof of this result occupies the rest of the paper.

\subsection{Superconnections and index}

We continue in the notations of the previous section.

A twisted superconnection $\A$ on the descent datum $(U_{\alpha}, \maL_{\alpha \beta}, \mu_{\alpha \beta \gamma}, \maE_{\alpha}, \varphi_{\alpha \beta})$  is a collection
$(\A_\alpha)_{\alpha\in \Lambda}$ of superconnections on the vector bundles $\pi_*\maE_\alpha$ over the open sets $U_\alpha$  such that when $U_{\alpha\beta} \not = \emptyset$,
\begin{equation}\label{suco}
(\pi_*\varphi_{\alpha\beta})^*   \A_\beta   = \A_\alpha \otimes \id + \id\otimes \nabla_{\alpha\beta}.
\end{equation}

We say that $\A$ is a Bismut superconnection if each $\A_{\alpha}$ is. Specifically we have
\[
\A_{\alpha}= D_{\alpha}+\nabla_{\alpha}^{\maH} -\frac{1}{4}c(T^{\maH})
\]
where $T^{\maH}$ is defined in \eqref{th}.

  The fact that $\nabla^{\maH}$ is a connection
on the descent datum, and $D\in \Psi^1(M|B, \maE)$, $c(T^{\maH}) \in \Omega^2(B, \Psi^0(M|B, \maE))$
implies that the conditions \eqref{suco} are satisfied.
From now on $\A$ will denote the  Bismut superconnection.

We will also consider the  rescaled Bismut superconnection  $\A_s$
defined by
\[
(\A_s)_{\alpha}= \A_{\alpha, s}: = s^{1/2}D_{\alpha}+\nabla_{\alpha}^{\maH} -\frac{1}{4}s^{-1/2}c(T^{\maH})
\]
where $s$ is either  a positive number or a multiple of the formal variable $u$. Denote by  $\theta^{\A_s}_{\alpha}$ the curvature of the rescaled superconnection
 $(\A_s)_{\alpha}^2$.
In particular we have forms $u\theta^{\A_{\alpha, u^{-1}}} \in \Omega^*(U_{\alpha}, \Psi(\pi^{-1}U_{\alpha}|U_{\alpha}, \maE_{\alpha}))[u^{1/2}]$.

\begin{proposition} There exists a form $u\theta^{\A_{u^{-1}}} \in \Omega^*(B, \Psi(M|B; \maE))[u^{1/2}]$ such that
\[\left.u\theta^{\A_{u^{-1}}}\right|_{U_{\alpha}}= u(\theta^{\A_{\alpha, u^{-1}}} + \pi^* \omega_{\alpha}).
\]
\end{proposition}

\begin{proof}
Recall that the curvature of $\nabla_{\alpha\beta}$ is equal to
 $\omega_\alpha- \omega_\beta$. Therefore from the equation \eqref{suco} we obtain
$$
\phi_{\alpha\beta} \left(u\theta^{\A_{\beta, u^{-1}}}\right)= u(\theta^{\A_{\alpha, u^{-1}}} + \pi^*\omega_\alpha - \pi^*\omega_\beta).
$$
The statement of the Proposition follows.
\end{proof}

Notice that $u\theta^{\A_{u^{-1}}} = D^2 +$ forms of degree $>0$. Therefore we can
define $e^{-\frac{u \theta^{\A_{u^{-1}}}}{2 \pi i}} \in \Omega^*(B, \Psi^{-\infty}(M|B; \maE))[u^{1/2}]$ by the usual Duhamel's formula.

 Note that the parity considerations as in finite dimensional case show that the coefficients for the nonintegral
powers of $u$ are odd with respect to the grading and hence have a vanishing supertrace.

Given a superconnection $\A= (\A_\alpha)_{\alpha\in \Lambda}$ on the $\Z_2$-graded  horizontally $\maL$-twisted Clifford module $\maE$ by the infinite-dimensional version of the  Proposition \ref{chernsuperconnection} the
the differential form
$$
\STr (e^{-\frac{u \theta^{\A_{u^{-1}}}}{2 \pi i}}) \in \Omega^*(B)[u],
$$
is closed with respect to the twisted de Rham differential $d_{\varOmega}$. The proof is
identical to the proof of the first part of the Proposition \ref{twchern}. The proof of the following Theorem is adapted from \cite{Gorokhovsky} and it uses ideas from
\cite{GorokhovskyLott}. The use of cyclic theory is inspired by \cite{Nistor}.

\begin{theorem} \label{superconnection index}
The following equality holds in the $H^{\bullet}_{\maL}(B)$:
\[
\Phi_{\nabla^{\maH}}(\ch(\ind D^+))=\left[\STr (e^{-\frac{u \theta^{\A_{u^{-1}}}}{2 \pi i}})\right]
\]
\end{theorem}

\begin{proof}

Consider the twisted bundle $\widetilde{\maE}=\maE\oplus \maE$ with the grading given by $\Gamma=\begin{bmatrix} \gamma &0\\ 0 &-\gamma \end{bmatrix}$. The algebra  $\Psi_\maL (M|B; \widetilde \maE)$ of operators on $\widetilde{\maE}$
is naturally $\Z_2$ graded. When discussing cyclic complexes of this algebra and its subalgebras we always consider it as
$\Z_2$-graded algebra. For an operator or (super) connection $K$ on $\maE$ set $\widetilde{K} = K \oplus K$; so for example $\widetilde{\A} =\A\oplus\A$, $\widetilde{\nabla}^{\maH}=\nabla^{\maH}\oplus \nabla^{\maH}$, etc.

Let $F=D(1+D^2)^{-1/2}$. Then it is immediate that $F\in \Psi^0_\maL (M|B; \maE)$ is odd with respect to  the grading $\Gamma$  and fiberwise elliptic.
Construct the invertible operator $U_D \in \Psi_\maL^{-\infty} (M|B ; \widetilde{\maE})$ by the same formula as before. Namely, 
 choose a parametrix
$R$ for $F$. Let $S_0=1-RF$, $S_1=1-FR$. Then set $U_{D}=\begin{bmatrix} S_0 &-(1+S_0)R\\F &S_1 \end{bmatrix}$. With such a choice the inverse is given by an explicit formula  $U_{D}^{-1}=\begin{bmatrix} S_0 &(1+S_0)R\\-F &S_1 \end{bmatrix}$. Set $P_D=U_D^{-1}\begin{bmatrix} 1_{\maE} &0\\0 &0 \end{bmatrix}U_D$.

The choices in the constructions can be made so that  we have $(P_D)^\pm=P_{D^\pm}$, see  Definition~\ref{defindex}.

Define the map $\Phi_{\widetilde{\nabla}^{\maH}} \colon CC_{\bullet}^- \left(\Psi_\maL^{-\infty} (M|B; \widetilde{\maE})\right) \to \left(\Omega^*(M)[u], d_{\varOmega}\right)$ by
$$
\Phi^k_{\widetilde{\nabla}^\maH}(A_0, \cdots, A_k) := \int_{\Delta^k} \STr\left( A_0 e^{-ut_0\frac{\widetilde{\theta}^{\maH}}{2 \pi i}} \widetilde{\partial}^\maH (A_1) e^{-ut_1\frac{\widetilde{\theta}^{\maH}}{2 \pi i}} \cdots e^{-ut_{k-1}\frac{\widetilde{\theta}^{\maH}}{2 \pi i}} \widetilde{\partial}^\maH (A_k)e^{-ut_k\frac{\widetilde{\theta}^{\maH}}{2 \pi i}}\right) dt_1\ldots dt_k
$$
for $A_0 \otimes\ldots \otimes A_k \in C_k\left(\Psi_\maL^{-\infty} (M|B; \widetilde{\maE})\right)$ and set again $\Phi _{\widetilde{\nabla}^\maH}= \sum_{k=0}^{\infty} \Phi^k_{\widetilde{\nabla}^\maH}$. Notice that
\begin{multline}\label{double}
\left[ \Phi _{\widetilde{\nabla}^\maH} \left(\Ch \left(P_D -\begin{bmatrix} 0 &0\\0 &1_{\maE} \end{bmatrix}\right)\right)\right]
=\\
 \Phi _{{\nabla}^\maH}\left(\ch\left(\ind D^+\right)\right)-\Phi _{{\nabla}^\maH}\left(\ch\left(\ind D^-\right)\right)  = 2 \Phi _{{\nabla}^\maH}\left(\ch\left(\ind D^+\right)\right)
\end{multline}

Replacing in the formulas above $\widetilde{\theta}^{\maH}$ by $ \widetilde{\theta}^{\A_{u^{-1}}}$ and $\widetilde{\partial}^{\maH}$ by $[\widetilde{\A}_{u^{-1}}, \cdot]$
we obtain the definition of the  morphism
$\Phi_{\widetilde{\A}} \colon CC_{\bullet}^{entire} \left(\Psi_\maL^{-\infty} (M|B; \widetilde{\maE})\right) \to \left(\Omega^*(M)[u], d_{\varOmega}\right)_{\bullet}$.

\begin{lemma}  The morphisms $\Phi_{\widetilde{\A}}$, $\Phi _{\widetilde{\nabla}^\maH} \colon CC_{\bullet}^{entire} \left(\Psi_\maL^{-\infty} (M|B; \widetilde{\maE})\right) \to \left(\Omega^*(B)[u], d_{\varOmega}\right)_{\bullet}$ are chain homotopic.
\end{lemma}
\begin{proof}
This follows from the explicit formula for the chain homotopy, see Proposition 5.6 of \cite{MathaiStevenson}.
This formula can be described as follows. Let $\A_{u^{-1}}(s)= s \A_{u^{-1}}+(1-s) \nabla^{\maH}$. Then we can write
\begin{multline*}
  \A_{u^{-1}}(s)=
su^{-1/2} D + \nabla^{\maH} - \frac{su^{1/2}}{4} c(T^{\maH}), \\ \theta^{\A_{u^{-1}}(s)} = \theta^{\maH} + s [\nabla^{\maH}, u^{-1/2} D- \frac{u^{1/2}}{4} c(T^{\maH})]+
s^2 (u^{-1/2} D- \frac{u^{1/2}}{4} c(T^{\maH}))^2\\ \text{ and } \frac{d}{ds} \A_{u^{-1}}(s) = u^{-1/2} D - \frac{u^{1/2}}{4} c(T^{\maH}).
\end{multline*}

Define $H_k \colon C_k(\Psi_{\maL}^{-\infty}(M|B, \maE \oplus \maE)) \to \Omega^*(B)[u]$ by
\begin{multline}
H_k(A_0, \ldots, A_k) = \int\limits_0^1 ds \left(
\sum  \limits_{m=0}^k (-1)^{m}\int_{\Delta^{k+1}}  dt_1\ldots dt_{k+1}  
 \STr\left(
 A_0 e^{-ut_0\frac{\widetilde{\theta}^{\A_{u^{-1}}(s)}}{2 \pi i}}    \right.  \right. \\  [\widetilde{\A}_{u^{-1}}(s), A_1] e^{-ut_1\frac{\widetilde{\theta}^{\A_{u^{-1}}(s)}}{2 \pi i}}\ldots
 [\widetilde{\A}_{u^{-1}}(s), A_m] e^{-ut_m\frac{\widetilde{\theta}^{\A_{u^{-1}}(s)}}{2 \pi i}}\
\widetilde{\frac{d}{ds} \A}_{u^{-1}} \ e^{-ut_{m+1}\frac{\widetilde{\theta}^{\A_{u^{-1}}(s)}}{2 \pi i}}
 \ldots  \\
 \left. \left. e^{-ut_{k-1}\frac{\widetilde{\theta}^{\A_{u^{-1}}(s)}}{2 \pi i}} [\widetilde{\A}_{u^{-1}}(s), A_k]e^{-ut_{k+1}\frac{\widetilde{\theta}^{\A_{u^{-1}}(s)}}{2 \pi i}}\right)\right)
\end{multline} 
and set $H = \sum H_k \colon CC_{\bullet}^{entire} \left(\Psi_\maL^{-\infty} (M|B; \widetilde{\maE})\right) \to \Omega^*(B)[u]$. Then
$$
\Phi_{\widetilde{\A}}-\Phi_{\nabla^{\maH}} = d_{\varOmega} \circ H + H\circ (b+uB).
$$
\end{proof}

Let $\maF $ be the algebra defined by
$$
\maF=\{ F \in \Psi_\maL^{0} (M|B; \widetilde{\maE})^{even}\ | \ [D, F] \in  \Psi_\maL ^{0} (M|B; \maE)\}.
$$
A simple modification of the argument in \cite{GS} as done in \cite{BenameurCarey}, cf. also \cite{Gorokhovsky}  shows that $\Phi_{\widetilde{\A}}$ extends to a morphism, also denoted  $\Phi_{\widetilde{\A}}\colon CC^{entire}_{\bullet} (\maF) \to \left(\Omega^*(M)[u], d_{\varOmega}\right)$ defined by the same formula.   Note that $U_D \in \maF$, $P_D \in \maF$ and
$\begin{bmatrix} 1_{\maE} &0\\0 &0 \end{bmatrix} \in \maF$.  

Recall now that inner automorphisms  act by identity on the entire cyclic homology, see e.g. \cite{Loday} 4.1.3 or \cite{Khalkhali}. It follows that the chains $\Ch(P_D)$ and $\Ch\left(\begin{bmatrix} 1_{\maE} &0\\0 &0 \end{bmatrix}\right)$ are homologous in $CC^{entire}_{\bullet} (\maF)$.

We therefore obtain
\begin{eqnarray*}
\Phi\left(\ch\left(\ind D^+\right)\right) & = & \frac{1}{2}\left[ \Phi _{\widetilde{\nabla}^\maH} \left(\Ch \left(P_D -\begin{bmatrix}  0&0\\0 &1_{\maE} \end{bmatrix}\right)\right)\right]\\
 & = &
\frac{1}{2}\left[ \Phi _{\widetilde{\A}} \left(\Ch \left(P_D  - \begin{bmatrix}  0&0\\0 &1_{\maE} \end{bmatrix} \right)\right)\right]\\
& = &  \frac{1}{2}\left[ \Phi _{\widetilde{\A}} \left(\Ch \left(P_D\right)\right) -\Phi _{\widetilde{\A}}\left(\Ch\left(\begin{bmatrix}  0&0\\0 &1_{\maE} \end{bmatrix}\right)\right)\right]\\
& = &
\frac{1}{2}\left[ \Phi _{\widetilde{\A}} \left(\Ch \left(\begin{bmatrix} 1_{\maE} &0\\0 &0 \end{bmatrix}\right)\right) \right]-\frac{1}{2}\left[ \Phi _{\widetilde{\A}}\left(\Ch\left(\begin{bmatrix}  0&0\\0 &1_{\maE} \end{bmatrix}\right)\right)\right]\\
 & = &
\left[\STr (e^{-\frac{u \theta^{\A_{u^{-1}}}}{2 \pi i}})\right].
\end{eqnarray*}
\end{proof}

\subsection{The local index theorem}

\begin{theorem}\label{local index}
Let  $D$ be a projective family of Dirac operators on a horizontally $\maL$-twisted Clifford  module $\maE$. We have the following equality of classes in $H^{\bullet}_{\maL}(B)$:
$$
  \left[\STr (e^{-\frac{u}{2\pi i}\theta^{\A_{u^{-1}}}}) \right]= \left[u^{-\frac{k}{2}}\int_{M|B}\widehat{A}\left(\frac{u}{2 \pi i}R^{M|B}\right)  \Ch_{\maL}(\maE/\maS)\right].
$$
\end{theorem}

\begin{proof}
Over each open set $U_{\alpha}$ we have $\theta^{\A_{u^{-1}}} = \A_{\alpha, u^{-1}}^2 + \pi^*\omega_\alpha$. Since $\A_{\alpha, u^{-1}}^2$ and $\pi^*\omega_\alpha$ commute, we have
$$
\STr e^{-\frac{u}{2\pi i}\theta^{\A_{u^{-1}}}}  = e^{-\frac{u}{2\pi i}\pi^*\omega_\alpha} \STr e^{-\frac{u}{2\pi i} \A_{\alpha, u^{-1}}^2}.
$$
 According to Bismut's local index theorem for families \cite{Bismut}
\[
\lim_{t\to 0}  \STr e^{-\frac{1}{2\pi i} \A_{\alpha, t}^2} =  \int_{\pi^{-1}U_{\alpha}|U_{\alpha}}\widehat{A}\left(\frac{1}{2 \pi i}R^{M|B}\right) \str_{\maE/\maS} e^{-\frac{1}{2\pi i} \theta^{\maE/\maS}_{\alpha}} .
\]
Moreover, by the result of Bismut and Fried \cite{BismutFreed}
\begin{equation}\label{BF}
 \STr e^{-\frac{1}{2\pi i} \A_{\alpha}^2} -\int_{\pi^{-1}U_{\alpha}|U_{\alpha}}\widehat{A}\left(\frac{1}{2 \pi i}R^{M|B}\right) \str_{\maE/\maS} e^{-\frac{1}{2\pi i} \theta^{\maE/\maS}_{\alpha}} = d \int_0^1 \xi_{\alpha}(t) dt
\end{equation}
where
\begin{equation}\label{defxi}
\xi_{\alpha}(t) = -\frac{1}{2\pi i}\STr \frac{d \A_{\alpha, t}}{dt} e^{-\frac{1}{2\pi i} \A_{\alpha, t}^2}
\end{equation}
is integrable at $0$.

Let $s>0$ and let $\delta_s^B$ be the operator on $\Omega^*(B)$ which multiplies the forms of degree $k$ by $s^{k/2}$.
Then $\delta_s^B\circ\left( \A_{\alpha, t}\right)\circ\delta_{s^{-1}}^B= s^{1/2}\A_{\alpha, t/s}$. We therefore obtain
\begin{eqnarray*}
\lim_{t\to 0}\STr e^{-\frac{s}{2\pi i} \A_{\alpha, t/s}^2} & = & \delta_s^B \left( \lim_{t\to 0} \STr e^{-\frac{1}{2\pi i}
\A_{\alpha, t}^2}\right)\\
& = &  \delta_s^B \int_{\pi^{-1}U_{\alpha}|U_{\alpha}}\widehat{A}\left(\frac{1}{2 \pi i}R^{M|B}\right) \str_{\maE/\maS} e^{-\frac{1}{2\pi i} \theta^{\maE/\maS}_{\alpha}}\\
& = &  s^{-\frac{k}{2}}\int_{\pi^{-1}U_{\alpha}|U_{\alpha}}\widehat{A}\left(\frac{s}{2 \pi i}R^{M|B}\right) \str_{\maE/\maS} e^{-\frac{s}{2\pi i} \theta^{\maE/\maS}_{\alpha}}
\end{eqnarray*}
where, as before, $k=\dim M - \dim B$. Since both sides are polynomials in $s$ we deduce that
\[
\lim_{t\to 0}\STr e^{-\frac{u}{2\pi i} \A_{\alpha, u^{-1}t}^2}=u^{-\frac{k}{2}}\int_{\pi^{-1}U_{\alpha}|U_{\alpha}}\widehat{A}\left(\frac{u}{2 \pi i}R^{M|B}\right) \str_{\maE/\maS} e^{-\frac{u}{2\pi i} \theta^{\maE/\maS}_{\alpha}}
\]
Multiplying both sides by $e^{-\frac{u}{2\pi i}\pi^*\omega_\alpha}$ we obtain

\[
\left.\lim_{t\to 0}\STr \exp\left(-\frac{u}{2\pi i} \theta^{\A_{u^{-1}t}}\right) \right|_{U_{\alpha}} =u^{-\frac{k}{2}}\int_{\pi^{-1}U_{\alpha}|U_{\alpha}}\widehat{A}\left(\frac{u}{2 \pi i}R^{M|B}\right)  \Ch_{\maL}(\maE/\maS)
\]
and therefore
\[
 \lim_{t\to 0}\STr \exp\left(-\frac{u}{2\pi i} \theta^{\A_{u^{-1}t}}\right)  =u^{-\frac{k}{2}}\int_{M|B}\widehat{A}\left(\frac{u}{2 \pi i}R^{M|B}\right)  \Ch_{\maL}(\maE/\maS).
\]

Moreover by \eqref{BF}, we have
\begin{equation}\label{BF2}
 \STr e^{-\frac{u}{2\pi i} \A_{\alpha, u^{-1}}^2}-u^{-\frac{k}{2}}\int_{\pi^{-1}U_{\alpha}|U_{\alpha}}\widehat{A}\left(\frac{u}{2 \pi i}R^{M|B}\right) \str_{\maE/\maS} e^{-\frac{u}{2\pi i} \theta^{\maE/\maS}_{\alpha}} =ud \int_0^1 u^{-\frac{1}{2}}\delta_u^B \xi_{\alpha}(t)dt
\end{equation}
Here the right hand side is defined as follows.
Write $\xi_{\alpha}^{[l]} \in \Omega^l(B)$ for the component of degree $l$ of $\xi_{\alpha}$ from the equation \eqref{defxi}. Then we have
\[
u^{-\frac{1}{2}} \delta_u^B\xi_{\alpha}(t)= \sum u^{\frac{l-1}{2}}\xi_{\alpha}^{[l]}(t)= \frac{1}{2\pi i}
 \STr \frac{d \A_{\alpha, u^{-1}t}}{dt} \exp \left({-\frac{u}{2\pi i} \A_{\alpha, u^{-1}t}^2} \right)
 \in \Omega^*(B)[u].
\]
Multiplying the  identity \eqref{BF2} by $e^{-\frac{u}{2\pi i}\pi^*\omega_\alpha}$ and using the equality
\[
e^{-\frac{u}{2\pi i}\pi^*\omega_\alpha} u d(\cdot)= (ud+u^2\varOmega)(e^{-\frac{u}{2\pi i}\pi^*\omega_\alpha} \cdot)
\]
we obtain
\[
\left. \STr \exp\left(-\frac{u}{2\pi i} \theta^{\A_{u^{-1}}}\right) \right|_{U_{\alpha}} -u^{-\frac{k}{2}}\int_{\pi^{-1}U_{\alpha}|U_{\alpha}}\widehat{A}\left(\frac{u}{2 \pi i}R^{M|B}\right)  \Ch_{\maL}(\maE/\maS)=
(ud +u^2 \varOmega) \Xi_{\alpha}
\]
 where $\Xi_{\alpha}= \int_0^1 \frac{1}{2\pi i}
 \STr \frac{d \A_{\alpha, u^{-1}t}}{dt} \exp \left(-\frac{u}{2\pi i}  \left.\theta^{\A_{ u^{-1}t}}\right|_{U_{\alpha}} \right)dt$.

From the equations \eqref{suco} it follows that over $U_{\alpha \beta}$ we have $\frac{d \A_{\alpha, u^{-1}t}}{dt} = \phi_{\alpha \beta} \left(\frac{d \A_{\beta, u^{-1}t}}{dt} \right) $. Hence there exists a form $\dot{\A}_t\in u^{-1/2}\Omega^*(B; \Psi_\maL (M|B, \maE))[u]$ such that $\left. \dot{\A}_t \right|_{U_{\alpha}}
= \frac{d \A_{\alpha, u^{-1}t}}{dt}$. Therefore setting $\Xi = \int_0^1 \frac{1}{2\pi i}
 \STr \dot{\A}_t \exp \left(-\frac{u}{2\pi i}  \theta^{\A_{u^{-1}t}} \right)dt \in \Omega^*(B)[u]$
  we can write
\[
\STr \exp\left(-\frac{u}{2\pi i} \theta^{\A_{u^{-1}}}\right)-u^{-\frac{k}{2}}\int_{M|B}\widehat{A}\left(\frac{u}{2 \pi i}R^{M|B}\right)  \Ch_{\maL}(\maE/\maS)
  = d_{\varOmega}\Xi
\]
and the statement of the Theorem follows.
\end{proof}

Combining results of the Theorems \ref{superconnection index} and \ref{local index} we obtain the
 main theorem of this paper

\begin{thm} \ Let  $D$ be a projective family of Dirac operators on a horizontally
$\maL$-twisted Clifford  module $\maE$ on a  fibration  $\pi:M\to B$. Then the following equality holds
in $H^{\bullet}_{\maL}(B)$:
$$
\left[\Phi_{\nabla^{\maH}} (\ch (\ind D^+))\right]  =
\left[u^{-\frac{k}{2}}\int_{M|B}\widehat{A}\left(\frac{u}{2 \pi i}R^{M|B}\right)  \Ch_{\maL}(\maE/\maS)\right],
$$
where $k=\dim M -\dim B$ is the dimension of the fibers.
\end{thm}


\begin{thebibliography}{9999}

\bibitem{AtiyahSinger4} Atiyah, M. F. and Singer, I. M. {\em The index of elliptic operators. IV.}  Ann. of Math. (2)  93  1971 119--138.
\bibitem{BenameurCarey}  Benameur, M.-T. and Carey A. {\em On the analyticity of the bivariant JLO cocycle.}  Electron. Res. Announc. Math. Sci.  16  (2009), 37--43.
\bibitem{BenameurHeitsch}  Benameur, M.-T. and Heitsch, J. L. {\em Index theory and non-commutative geometry. II. Dirac operators and index bundles.}  J. K-Theory  1  (2008),  no. 2, 305--356.
\bibitem{bgv} Berline, N., Getzler, E. and Vergne, M. {\em Heat kernels and Dirac operators.} Corrected reprint of the 1992 original. Grundlehren Text Editions. Springer-Verlag, Berlin, 2004.
\bibitem{Bismut} Bismut, J.-M., {\em The Atiyah-Singer index theorem for families of Dirac operators: two heat equation proofs}.  Invent. Math.  83  (1985),  no. 1, 91--151.
\bibitem{BismutFreed} Bismut, J.-M. and Freed, D. S., {\em The analysis of elliptic families. II. Dirac operators, eta invariants, and the holonomy theorem.}  Comm. Math. Phys.  107  (1986),  no. 1, 103--163.
\bibitem{BottTu} Bott, R. and Tu, L. W. {\em Differential forms in algebraic topology.} Graduate Texts in Mathematics, 82. Springer-Verlag, New York-Berlin, 1982.
\bibitem{BM} Breen, L. and Messing, W.  {\em Differential Geometry of Gerbes.} Advances in Math.
198, 732–846 (2005).
\bibitem{BGNT1} Bressler, P., Gorokhovsky, A.,  Nest, R. and Tsygan B., {\em Deformations of gerbes on smooth manifolds}.  $K$-theory and noncommutative geometry,  349--392, EMS Ser. Congr. Rep., European Mathematical Society, Z\"urich 2008.
\bibitem{BGNT2} Bressler, P., Gorokhovsky, A.,  Nest, R. and Tsygan B., {\em Deformation quantization of gerbes.}  Adv. Math.  214  (2007),  no. 1, 230--266. 
\bibitem{Brylinski} Brylinski, J.-L. {\em Loop spaces, characteristic classes and geometric quantization.} Modern Birkhäuser Classics. Birkhäuser Boston, Inc., Boston, MA, 2008.
\bibitem{CW} Carey A. and  Wang B.-L.,  {\em Riemann-Roch and index formulae
in twisted K-theory}.  to appear in Proceedings of NSF/CBMS Conference
on ``Topology, C*-algebras, and String Duality''.
\bibitem{CMW} Carey, A.,  Mickelsson J. and Wang B.-L.,  {\em Differential
Twisted K-theory and its Applications}, J. Geom. Phys. 59 (2009), no.
5, 632--653.
\bibitem{CaW} Carrillo Rouse P. and Wang B.-L., {\em Twisted longitudinal index theorem for foliations and wrong way functoriality }, arXiv:1005.3842.
\bibitem{ConnesJLO} Connes, A. {\em Entire cyclic cohomology of Banach algebras and characters of $\theta$-summable Fredholm modules.}  $K$-Theory  1  (1988),  no. 6, 519--548.
\bibitem{ConnesIHES} Connes, A. {\em Noncommutative differential geometry.}  Inst. Hautes \'Etudes Sci. Publ. Math.  No. 62  (1985), 257--360.
\bibitem{ConnesBook} Connes, A. {\em Noncommutative geometry.} Academic Press, Inc., San Diego, CA, 1994.
\bibitem{DixmierDouady} Dixmier, J. and Douady, A., {\em
Champs continus d'espaces hilbertiens et de $C^{\ast} $-alg\`ebres.}
Bull. Soc. Math. France 91 1963 227--284.
\bibitem{DK} Donovan, P. and Karoubi, M., {\em
 Graded Brauer groups and $K$-theory with local coefficients.} 
Inst. Hautes \'Etudes Sci. Publ. Math. No. 38 (1970), 5--25
\bibitem{GS} Getzler, E. and  Szenes, A. {\em On the Chern character of a theta-summable Fredholm module.}  J. Funct. Anal.  84  (1989),  no. 2, 343--357.
\bibitem{Gorokhovsky1} Gorokhovsky, A. {\em Characters of cycles, equivariant characteristic classes and Fredholm modules.} Comm. Math. Phys. 208 (1999), no. 1, 1--23.
\bibitem{Gorokhovsky} Gorokhovsky, A. {\em Bivariant Chern character and longitudinal index.}  J. of Funct. Analysis 237  (2006),  105-134.
\bibitem{GorokhovskyLott} Gorokhovsky, A. and Lott, J. {\em Local index theory over \`{e}tale groupoids.}  J. Reine Angew. Math.  560  (2003), 151--198.
\bibitem{Hitchin} Hitchin, N. J. {\em  The moduli space of complex Lagrangian submanifolds.}  Surveys in differential geometry,  327--345, Surv. Differ. Geom., VII, Int. Press, Somerville, MA, 2000.
\bibitem{Karoubi} Karoubi, M.
{\em Twisted $K$-theory---old and new.}, $K$-theory and noncommutative geometry, 117--149, EMS Ser. Congr. Rep., European Mathematical Society, Z\"urich 2008.
 \bibitem{Khalkhali} Khalkhali, M. { On the entire cyclic cohomology of Banach algebras.}   Comm. Algebra 22 (1994), no. 14, 5861--5874.
\bibitem{Loday} Loday, J.-L., {\em Cyclic homology. } Grundlehren der Mathematischen Wissenschaften, 301. Springer-Verlag, Berlin, 1992.
\bibitem{MMS1} Mathai, V.; Melrose, R. B. and Singer, I M. {\em The index of projective families of elliptic operators.}  Geom. Topol.  9  (2005), 341--373.
\bibitem{MMS2} Mathai, V.; Melrose, R. B. and Singer, I M. {\em The index of projective families of elliptic operators: the decomposable case.} Asterisque, 327 (2009) 251-292
\bibitem{MathaiStevenson} Mathai, V. and Stevenson, D. {\em On a generalized Connes-Hochschild-Kostant-Rosenberg theorem.}  Adv. Math.  200  (2006),  no. 2, 303--335.
\bibitem{MeyerThesis} Meyer, R., {\em Local and analytic cyclic homology.} EMS Tracts in Mathematics, 3. European Mathematical Society (EMS), Zurich, 2007.
\bibitem{Murray} Murray, M. K. {\em Bundle gerbes.}  J. London Math. Soc. (2)  54  (1996),  no. 2, 403--416.
\bibitem{NestTsygan1} Nest R. and Tsygan B. {\em  Algebraic index theorem for families.}   Adv. Math. 113 (1995), no. 2, 151--205.
\bibitem{NestTsygan2} Nest R. and Tsygan B. {\em On the cohomology ring of an algebra. } Advances in geometry,337--370,
Progr. Math., 172, Birkh�user Boston, Boston, MA, 1999.
\bibitem{Nistor} Nistor, V. {\em Super-connections and non-commutative geometry.}  Cyclic cohomology and noncommutative geometry (Waterloo, ON, 1995),
115--136, Fields Inst. Commun., 17, Amer. Math. Soc., Providence, RI, 1997.
 \end{thebibliography}
\end{document}